\documentclass[oneside]{amsart}

\usepackage{hyperref}
\usepackage{tikz}
\usepackage{latexsym,amssymb,amsmath,bm,enumitem,verbatim,todonotes,stmaryrd}

\usepackage{tikz}
\usetikzlibrary{shapes.geometric}
\usepackage{mathdots}
\usetikzlibrary{arrows,calc}

\tikzstyle{edge}=[shorten <=1pt, shorten >=1pt, >=stealth, line width=1pt]
\tikzstyle{vertex}=[circle, fill=black, draw, minimum size=4pt, inner sep=0pt, outer sep=0pt]

\title[Large cardinals and compactness for list colourings]{Large cardinal characterizations via compactness for list colourings}

\author{Roman Feller}
\address{Institut f\"ur diskrete Mathematik und Geometrie\\
TU Wien\\
Wiedner Hauptstrasse 8-10/104\\
1040 Vienna\\
Austria}
\email{roman.feller@tuwien.ac.at}
\urladdr{https://orcid.org/0009-0008-4825-8381}

\author{Peter Holy}
\address{Institut f\"ur diskrete Mathematik und Geometrie\\
TU Wien\\
Wiedner Hauptstrasse 8-10/104\\
1040 Vienna\\
Austria}
\email{peter.holy@tuwien.ac.at}
\urladdr{https://orcid.org/0000-0001-8643-1845}

\subjclass[2020]{Primary 03E55, 03E75; Secondary 03C75, 05C63, 05C15}
\keywords{Compact Cardinal, Weakly Compact Cardinal, Strongly Compact Cardinal, List Colouring, Graph Colouring, Compactness}

\newcommand{\ZF}{\mathrm{ZF}}
\newcommand{\id}{\mathrm{id}}
\newcommand{\G}{\mathcal G}

\newcommand{\Ord}{\mathrm{Ord}}

\newcommand{\cG}{\mathcal G}
\newcommand{\cH}{\mathcal H}
\newcommand{\cP}{\mathcal P}
\newcommand{\cR}{\mathcal R}
\newcommand{\cS}{\mathcal S}
\newcommand{\cT}{\mathcal T}
\newcommand{\ZFC}{\mathrm{ZFC}}
\newcommand{\cof}{\mathrm{cof}}
\newcommand{\Ult}{\mathrm{Ult}}
\newcommand{\crit}{\mathrm{crit}}
\newcommand{\ot}{\mathrm{ot}}
\newcommand{\GCH}{\mathrm{GCH}}

\newtheorem{fact}{Fact}
\newtheorem{lemma}[fact]{Lemma}
\newtheorem{theorem}[fact]{Theorem}
\newtheorem{corollary}[fact]{Corollary}
\newtheorem*{claim}{Claim}

\newtheorem{observation}[fact]{Observation}
\newtheorem{proposition}[fact]{Proposition}

\theoremstyle{definition}
\newtheorem{question}[fact]{Question}

\newtheorem{definition}[fact]{Definition}

\hypersetup{hidelinks}

\begin{document}

\begin{abstract}
  We investigate compactness properties with respect to list colouring, a certain form of graph colouring, with infinitely many colours. We introduce a new hierarchy of compactness cardinals, that also includes some well-established large cardinal notions, and use it to show that various types of large cardinals, including weakly compact, strongly compact, and $\delta$-strongly compact cardinals, can be characterized in terms of compactness for list colouring. We also obtain lower bounds on the size of our newly introduced compactness cardinals. %We also show that $\delta^+$-strongly compact cardinals can be characterized in  terms of compactness for the existence of homomorphisms between certain types of relational structures.
\end{abstract}

\maketitle

\section{Introduction}

Given a set $\Omega$, we say that a graph $\G=\langle G,E\rangle$ with vertex set $G$ and edge relation $E$ is \emph{$\Omega$-colourable} (or \emph{chromatically $\Omega$-colourable}) if there is a function $c\colon G\to\Omega$ such that $c(g)\ne c(h)$ whenever $E(g,h)$ holds. \emph{Compactness for $\Omega$-colouring} is the statement that for every graph $\G$, if every finite subgraph of $\G$ is $\Omega$-colourable, then~$\G$ itself is $\Omega$-colourable. Classical results of De Bruijn and Erd\H{o}s~\cite{bruijnerdos} and L\"auchli~\cite{laeuchli} (see for example \cite{cowen} or \cite{ktv} for presentations of different arguments) show that compactness for $3$-colouring is equivalent to the weak version of the axiom of choice that is the ultrafilter lemma, over the base system~$\ZF$. 
% Given these results, an obvious question is whether this can somehow be extended to situations when the relevant structure $\mathcal A$ is infinite (see Definition \ref{def:kappahomcom}). 
% Simply allowing for $\mathcal A$ to be countable makes the above false, by well-known standard arguments.
%
% \begin{observation}\label{observation:infinitecounterexample}
%   For every finite (or even countable) substructure $X$ of $(\mathbb R,<)$, where $<$ denotes the natural ordering of $\mathbb R$, there is a homomorphism from $X$ to $(\mathbb Q,<)$, however there is no homomorphism from $(\mathbb R,<)$ to $(\mathbb Q,<)$, for such a homomorphism would in particular need to be injective.
% \end{observation}
%
% We will however see that a more careful adaption to infinite structures $\mathcal A$ turns out to be interesting. 
%Let us turn back to graph colouring. %As usual, if $\G$ or~$\cH$ denote graphs, we let $G$ and $H$ denote their set of vertices.
In this paper, we investigate variants of the above when allowing for infinitely many colours, over the base system $\ZFC$. Those variants will turn out to be directly connected to certain types of large cardinals, which we will introduce next. Let $\delta$ be an infinite cardinal. $\mathcal L_{\delta,\omega}$ denotes the logic extending first-order logic by allowing for conjunctions and disjunctions of arbitrary size less than $\delta$, however with formulas using only finitely many free variables. A theory (in any logic) is said to be \emph{${<}\kappa$-satisfiable} if every subset of this theory of size less than $\kappa$ has a model.

\begin{definition}\label{def:compact}
  Let 
  %\omega<\delta\le\kappa\le\lambda$ be cardinals, 
  $\delta\le\kappa\le\lambda$ be uncountable cardinals,
  with $\delta$ regular. Then, we say that $\kappa$ is \emph{$(\delta,\lambda)$-compact} if every ${<}\kappa$-satisfiable $\mathcal L_{\delta,\omega}$-theory in a language of size at most~$\lambda$ is satisfiable. We say that $\kappa$ is \emph{$(\delta,{<}\lambda)$-compact} if $\kappa$ is $(\delta,\bar\lambda)$-compact whenever $\bar\lambda<\lambda$, and we also allow for $\lambda=\Ord$ in this case. % different notion that provides a hierarchy leading up to strong compactness is that of \emph{$\gamma$-compactness} (see for example \cite[Page 307]{kanamori}). It can be seen from Proposition \ref{proposition:filterextensionchar} below and \cite[Theorem 22.17]{kanamori} that there is some connection between these cardinals and the hierarchy of $(\kappa,\lambda)$-compact cardinals for increasing $\lambda$.
 %\roman{Astethic question: the paper by Boney and Spencer uses $\infty$ instead of $\Ord$}
\end{definition}
%\todo[inline]{Roman: Other than convention, is there a good reason for $\kappa$ to be required to be uncountable?\\Peter: No, not really. I guess one wants to have the assumption of the existence of a delta-strongly cardinal to have nontrivial consistency strength beyond ZFC. And if we allow for $\omega$, then there is no consistency strength of course. Or, phrasing it differently, one wants to exclude the trivial case of $\omega$ right away.}

Note that the above property of $\kappa$ becomes stronger as $\delta$ or $\lambda$ increase, and that if $\kappa$ is $(\delta,\lambda)$-compact, then every $\kappa^*\in[\kappa,\lambda]$ is $(\delta,\lambda)$-compact as well. We will take a closer look at $(\delta,\lambda)$-compactness in Section \ref{section:compact}, and in particular show the following:

\begin{theorem}\label{theorem:compactlarge}
  For any cardinals $\delta\le\kappa$ with $\delta$ regular and uncountable, if $\kappa$ is $(\delta,\kappa)$-compact, then $\kappa\ge\aleph_\delta$.
\end{theorem}

In their \cite{boneyunger}, Boney and Unger introduce what they call $(\delta,\lambda)$-strongly compact cardinals, and we will show in Section \ref{section:more} that these are closely connected to our notion of $(\delta,\lambda)$-compact cardinals.

\medskip

However, the main point of introducing the above hierarchy of compactness properties is to unify the arguments that will follow, for the three well-known instances below -- note that:
\begin{itemize}
  \item $\kappa$ is \emph{weakly compact} if $\kappa$ is $(\kappa,\kappa)$-compact.
  \item $\kappa$ is \emph{$\delta$-strongly compact} if $\kappa$ is $(\delta,{<}\Ord)$-compact.
  \item $\kappa$ is \emph{strongly compact} if $\kappa$ is $(\kappa,{<}\Ord)$-compact.
\end{itemize}
%\footnote{We will establish a number of general results for $(\delta,\lambda)$-compact cardinals, but no effort was made to further investigate their properties in general, which we leave for possible future investigation.}

The fact that compact cardinals yield compactness properties for graph colouring is well-known. The following proposition is verified as an example in~\cite{bm}, and we will obtain a more general result in Proposition \ref{proposition:extensiontocompactness}.

\begin{definition}
  Let $\delta<\kappa$ be cardinals with $\kappa$ infinite. We say that \emph{$\kappa$-compactness for $\delta$-colouring} holds if whenever $\G$ is a graph with the property that every ${<}\kappa$-size subgraph $\cH$ of~$\G$ has a $\delta$-colouring, also $\G$ itself has a $\delta$-colouring.%\footnote{A $\delta$-colouring of a graph $\G=\langle G,E\rangle$ is a function $c\colon G\to\delta$ such that $c(g)\ne c(h)$ whenever $E(g,h)$ holds. In the literature, this is also referred to as a \emph{chromatic} $\delta$-colouring.}
\end{definition}

\begin{proposition}[Bagaria, Magidor]\label{prop:bm}
  If a cardinal $\kappa$ is $\omega_1$-strongly compact, then $\kappa$-compactness for $\omega$-colouring holds.
\end{proposition}

However, compactness for graph colouring can consistently hold at small cardinals, at least for relatively small graphs. Shelah has shown \cite[Theorem 5.2]{shelah2} that starting from a supercompact cardinal, using a model constructed by Ben-David and Magidor \cite{bdm}, it is consistent that $\aleph_\omega$-compactness holds for $\omega_n$-colouring for graphs of size $\aleph_{\omega+1}$ whenever $0<n<\omega$, while the $\GCH$ holds. In a note on his webpage \cite{unger}, Unger adapts this to show that starting from a supercompact cardinal, it is also consistent that $\aleph_{\omega_1}$-compactness holds for $\omega_{\alpha+1}$-colouring for graphs of size $\aleph_{\omega_1+1}$ whenever $0<\alpha<\omega_1$, while the $\GCH$ holds below $\aleph_{\omega_1}$.

\medskip

What seems to be very difficult is to obtain any kind of reversals (i.e., implications of compactness properties for graph colouring) beyond L\"auchli's above-mentioned classic result. A partial reversal by Shelah \cite{shelah} is that if $\kappa>\omega_1$ is a regular cardinal with a non-reflecting stationary subset of $E^\kappa_\omega$, and $\mu^\omega\le\kappa$ whenever $\mu<\kappa$, then $\kappa$-compactness for $\omega$-colouring fails for a graph of size $\kappa$.\footnote{Note that if $\kappa$ is weakly compact or $\omega_1$-strongly compact, well-known standard arguments show that every stationary subset of $E^\kappa_\omega$ reflects.} 

\medskip

In this paper, we obtain reversals that allow us to retrieve optimal large cardinal strength from compactness properties with respect to \emph{list colouring}, i.e., colouring with respect to a function (or \emph{list}) $L$ that assigns to each node of a graph a set of allowed colours. These reversals will also contrast the above-mentioned result of Shelah \cite[Theorem 5.2]{shelah2}, showing that compactness for graph colouring and compactness for list colouring behave fairly differently.

\begin{definition}
  Given a graph $\mathcal G=\langle G,E\rangle$ and a function $L$ with domain $G$, we say that a function $c$ with domain $G$ is an \emph{$L$-colouring} of $G$ if
\begin{itemize}
    \item $\forall g\in G\ c(g)\in L(g)$ and
    \item $\forall g,h\in G\ [E(g,h)\to c(g)\ne c(h)]$.\footnote{Note that $\delta$-colouring corresponds to $L$-colouring when $L$ is constant with value $\delta$.}
\end{itemize}
\end{definition}

In order to be able to talk about induced colourings of subgraphs more easily, we will also allow for the domains of $L$ and of $c$ to be (proper) supersets of $G$ in the above.
Corresponding compactness properties are given by the following:

\begin{definition}
    Let $\delta$ and $\kappa$ be cardinals, with $\kappa$ infinite.
    \begin{itemize}
        \item 
    We say that \emph{$\kappa$-compactness for $\delta$-list-colouring} holds if whenever $\G$ is a graph and $L \colon G \to \cP(\delta)$ is such that every ${<}\kappa$-size subgraph $\cH$ of $\G$ has an $L$-colouring, also $\G$ itself has an $L$-colouring.
        \item
    We say that \emph{$\kappa$-compactness for ${<}\delta$-list-colouring} holds if whenever $\G=\langle G,E\rangle$ is a graph and $L \colon G \to [\delta]^{<\delta}$ is such that every ${<}\kappa$-size subgraph $\cH$ of $\G$ has an $L$-colouring, then also $\G$ itself has an $L$-colouring.
  \end{itemize}
\end{definition}

Our main results will establish the following equivalences:

\begin{theorem}\label{th:main}
  Assume that $\delta\le\kappa\le\lambda$ are infinite cardinals, with $\delta$ regular. Then:
  \begin{enumerate}
      \item If $\delta<\kappa$ and $\lambda^\delta=\lambda$, then $\kappa$ is $(\delta^+,\lambda)$-compact if and only if $\kappa$-compactness for $\delta$-list-colouring for graphs of size at most $\lambda$ holds.
      \item If $\delta>\omega$ and $\lambda^{<\delta}=\lambda$, then $\kappa$ is $(\delta,\lambda)$-compact if and only if $\kappa$-compactness for ${<}\delta$-list-colouring for graphs of size at most $\lambda$ holds.
  \end{enumerate}  
  In particular, we obtain the following:
  \begin{itemize}
    \item An uncountable cardinal $\kappa$ is $\delta^+$-strongly compact if and only if $\kappa$-compact\-ness for $\delta$-list-colouring holds.
    \item An uncountable cardinal $\kappa$ is $\delta$-strongly compact if and only if $\kappa$-compact\-ness for ${<}\delta$-list-colouring holds.
    \item An uncountable cardinal $\kappa$ is weakly compact if and only if $\kappa=\kappa^{<\kappa}$ and $\kappa$-compactness for ${<}\kappa$-list-colouring holds for graphs of size at most $\kappa$.
    \item An uncountable cardinal $\kappa$ is strongly compact if and only if $\kappa$-compactness for ${<}\kappa$-list-colouring holds.\footnote{Note that $\kappa$-compactness for $\kappa$-list-colouring is inconsistent, as is witnessed by the complete graph on $\kappa^+$ and the constant function with value $\kappa$.}
  \end{itemize}
\end{theorem}

% We then proceed to investigate the hierarchy of $(\delta,\lambda)$-compact cardinals, in particular verifying the following:

% \begin{theorem}\label{theorem:compactcards}
% \begin{itemize}
%   \item If $\delta\le\kappa\le\lambda$ are uncountable cardinals with $\delta$ regular and with $\lambda^{<\kappa}=\lambda$, then $\kappa$ is $(\delta,\lambda)$-compact if and only if there is a fine, ${<}\delta$-complete ultrafilter on $\cP_\kappa(\lambda)$.\footnote{The latter property defines what is called a \emph{$(\delta,\lambda)$-strongly compact} cardinal in \cite{boneyunger}.}
%   \item If $\kappa$ is the least $(\omega_1,\kappa)$-compact cardinal\footnote{More formally correct, we should say that $\kappa$ is the least cardinal $\nu$ that is $(\omega_1,\nu)$-compact.} and satisfies $\kappa^\omega=\kappa$, then $\kappa$ is a limit cardinal. In particular, we obtain the following:
%   \begin{itemize}
%     \item If $n<\omega$, $\omega_n$ is not $(\omega_1,\omega_n)$-compact.\footnote{Note that the strength of the property of $\kappa$ being $(\delta,\kappa)$-compact increases as $\delta$ increases, so actually, $\omega_n$ is not $(\delta,\omega_n)$-compact whenever $\omega_1\le\delta\le\omega_n$.}
%     \item If the $\GCH$ holds, the least $(\omega_1,\kappa)$-compact cardinal $\kappa$ is a limit cardinal.
%   \end{itemize}
%   \item $\aleph_\omega$ is not $(\omega_1,\omega_\omega)$-compact.
% \end{itemize}
% \end{theorem}

Together with Theorem \ref{theorem:compactlarge}, Theorem \ref{th:main} yields some immediate consequences on compactness for list colouring: for example, under the $\GCH$, $\kappa^+$-compactness (and hence also $\kappa$-compactness) for $\omega_n$-list-colouring for graphs of size $\kappa^+$ fails whenever $n<\omega$ and $\kappa<\aleph_{\omega_{n+1}}$. This contrasts Shelah's result \cite[Theorem 5.2]{shelah2}.

\medskip

Let us provide an overview of the contents of our paper. In Section \ref{section:filterextension}, we introduce a family of filter extension properties, and show that they are closely related to the family of compactness properties for cardinals introduced in Definition~\ref{def:compact}, thus generalizing a classic result of Bell \cite{bell}. This relationship will be important for the verification of our main results. In Section \ref{section:homomorphismcompactness}, we introduce a family of homomorphism-compactness properties\footnote{This paper was inspired by results that link the axiomatic strength of homomorphism-compactness for a finite target structure $\mathcal A$ over $\ZF$ to the computational complexity of the decision problem $\mathrm{CSP}(\mathcal A)$; see~\cite{ktv,RTW,Tardif}.
Here, the latter problem---the \textit{constraint satisfaction problem} of~$\mathcal A$---consists of deciding whether a given finite input structure $\mathcal X$ homomorphically maps to $\mathcal A$.} generalizing the compactness for graph colouring and list colouring,
and then provide the results that will yield the forward implications in Theorem~\ref{th:main}. In Section \ref{section:graphconstructions}, we construct certain families of infinite graphs that we will then make use of in the argument for the reverse implications in Theorem \ref{th:main}, which we provide in Section \ref{section:main}. We also provide a version of Theorem \ref{th:main} with respect to homomorphism-compactness in that section. In Section~\ref{section:compact}, we verify Theorem \ref{theorem:compactlarge}. In Section \ref{section:more}, we investigate the relationship between $(\delta,\lambda)$-compact and $(\delta,\lambda)$-strongly compact cardinals, and provide a further characterization of the former. Finally, we argue that list colouring is a natural generalization of graph colouring (with finitely many colours) to the case of infinitely many colours in Appendix \ref{natural}.
%, and we briefly mention some connections to questions about constraint satisfaction problems in Appendix \ref{section:csp}, which were the initial motivation for the work in this paper.\peter{Probably needs to be adapted a bit.}

\medskip

Most of our paper should be readable with only very basic knowledge of set theory. However, Section \ref{section:more} (which is not overly relevant to the other parts of our paper) is slightly more aimed at a set theoretic audience, and some basic knowledge about the interplay between ultrafilters and elementary embeddings will be helpful in order to follow the arguments in this part of our paper.

\section{Filter Extension Properties}\label{section:filterextension}

\begin{definition}
  Let $X$ be a set and $F\subseteq E\subseteq\mathcal P(X)$, with $E$ closed under complements in $X$, and $X \in E$. %
  Let $\kappa$ be an infinite cardinal. We say that $F$ is a \emph{${<}\kappa$-complete filter (on $E$)}\footnote{Note that $X=\bigcup E$.} if
  \begin{itemize}
    \item $X\in F$, $\emptyset\not\in F$,
    \item $\forall Y, Z\in E\ [Y\supseteq Z\in F\to Y\in F]$, and
    \item whenever $(Y_i)_{i<\lambda}$ is a family of elements of $F$ with $\lambda<\kappa$, we have \[\bigcap_{i<\lambda}Y_i\ne\emptyset.\footnote{More commonly, the final requirement on $F$ would be that $\bigcap_{i<\lambda}Y_i\in F$. However, for the case of singular $\kappa$, this seems to be an overly narrow restriction for many applications. Another common variant would be to require that $\bigcap_{i<\lambda}Y_i$ has the same cardinality as $X$. The definition that we chose is the most useful one for our present purposes.}\]
  \end{itemize}
  
  We call $F$ a \emph{filter} if it is a ${<}\omega$-complete filter.
  We say that a filter $F$ \emph{measures}~$E$ if either $Y\in F$ or $X\setminus Y\in F$ whenever $Y\in E$.
  %If $X\subseteq\mathcal P(A)$ in the above, we say that a filter $F$ is \emph{fine} if $\{x\in X\mid a\in x\}\in F$ whenever $a\in A$.
  \end{definition}

\begin{definition}\label{definition:filterextension}
    %Let $\omega<\delta\le\kappa\le\lambda$ be 
    Let $\delta\le\kappa\le\lambda$ be infinite
    cardinals, with $\delta$ regular.
    \begin{itemize} 
      \item We say that~$\kappa$ has the \emph{$(\delta,\lambda)$-filter extension property} if whenever~$X$ is a set, $E\subseteq\cP(X)$ is of size at most $\lambda$ and closed under complements (in $X$) with $X\in E$, and $F\subseteq E$ is a ${<}\kappa$-complete filter, then~$F$ can be extended to a ${<}\delta$-complete filter that measures $E$.
      \item The \emph{$(\delta,{<}\lambda)$-filter extension property} is the above property for $E$ of size less than $\lambda$.
    \end{itemize}
\end{definition}

%Note that if $X=\kappa$ is a cardinal, $\bar F\supseteq\{\kappa\setminus\alpha\mid\alpha<\kappa\}$, and $F$ is a ${<}\omega$-complete filter on $X$ that extends $\bar F$, then $F$ is uniform: If $F$ contained some $Y\subseteq\alpha<\kappa$, then $Y\cap(\kappa\setminus\alpha)=\emptyset$, which contradicts the ${<}\omega$-completeness of $F$. In particular, if $X=\kappa$ is a cardinal in the above definition, we can always assume that the filters $U$ that we obtain are uniform.

%\medskip

Our next proposition extends a result of Bell \cite[Main Theorem]{bell}: we extend his result to include singular cardinals $\kappa$,\footnote{This may be folklore knowledge, but we could not find a written account for the case of singular cardinals, and moreover the set theoretic literature seems to be unspecific about what notion of ${<}\kappa$-completeness of a filter should be used in this case, while seemingly, only the less standard notion that we introduced above seems to make the argument work in the case of singular cardinals $\kappa$.} and also provide a stratification of his result with respect to the parameter $\lambda$ that provides either the size of the languages for which we have certain forms of compactness, or the size of the families of sets that we would like to be able to measure.
This requires some nontrivial extra care, in particular in the proof of the second part of the following proposition:

\begin{proposition}\label{proposition:filterextensionchar}
    %Let $\omega<\delta\le\kappa\le\lambda$ be 
    Let $\delta\le\kappa\le\lambda$ be uncountable 
    cardinals, with $\delta$ regular.
    \begin{itemize} 
        \item If $\kappa$ is $(\delta,\lambda)$-compact, then $\kappa$ has the $(\delta,\lambda)$-filter extension property.
      %  \item If $\kappa$ has the $(\delta,\lambda)$-filter property, then there is a fine, ${<}\delta$-complete ultrafilter on $[\lambda]^{<\kappa}$.
        \item If $\kappa$ has the $(\delta,\lambda^{<\delta})$-filter extension property, then $\kappa$ is $(\delta,\lambda)$-compact.
    \end{itemize}
\end{proposition}
\begin{proof}
    First, assume that $\kappa$ is $(\delta,\lambda)$-compact and let $F\subseteq E\subseteq\cP(X)$, with $F$ a ${<}\kappa$-complete filter, and $E$ of size at most $\lambda$, closed under complements and with $X\in E$. We ought to extend $F$ to a ${<}\delta$-complete filter $U$ that measures~$E$. 
    Let $\Sigma$ be the $\mathcal L_{\delta,\omega}$-theory of the structure \[\mathfrak A= (X,(P)_{P\in E} )\] with a unary relation $P$ corresponding to every $P\in E$, together with the set of sentences $\{P(c)\mid P\in F\}$, where $c$ is a new constant symbol. The language that we use here has size at most $\lambda$. As $F$ is ${<}\kappa$-complete, $\Sigma$ is ${<}\kappa$-satisfiable, as we may witness any collection $\Sigma'\subseteq\Sigma$ of ${<}\kappa$-many sentences by picking $c$ to be an element of the (nonempty) intersection of the ${<}\kappa$-many elements $P$ of $F$ for which $P(c)$ is in~$\Sigma'$. This means that by our assumption, $\Sigma$ is satisfiable, so let %without loss of generality,
    \[\mathfrak B= (Y,(Q_P)_{P\in E},d)\models\Sigma,\] with $d$ the interpretation of $c$ in~$\mathfrak B$. Let $U=\{P\in E\mid \mathfrak B\models Q_P(d)\}$. 
    If $P\in U$ and $R\supseteq P$ with $R\in E$, then $Q_P(d)$ implies $Q_R(d)$, and hence $R\in U$. If $P\in E$ and $R=X\setminus P$, then $Q_R=Y\setminus Q_P$, hence either $Q_P(d)$ or $Q_R(d)$ holds, so either~$P$ or~$R$ is in $U$. If $P\in F$, then $Q_P(d)$ holds, so $P\in U$. Finally, whenever $(P_i)_{i<\bar\delta}$ is a family of elements of $U$ for some $\bar\delta < \delta$, $\mathfrak B$ satisfies the $\mathcal L_{\delta,\omega}$-formula $\exists x \bigwedge_{i < \bar\delta} Q_{P_i}(x)$ as it is witnessed by $d$, hence $\bigcap_{i < \bar\delta} P_i$ is non-empty, yielding that $U$ is a ${<}\delta$-complete filter that measures $E$.
  %But also, $\mathfrak B$ satisfies the $\mathcal L_{\delta^+,\omega}$-formula stating that $d\in\bigcap_{i<\delta} Q_{P_i}$ whenever all $P_i\subseteq X$ are in $U$, and hence the $L_{\delta^+,\omega}$-formula stating that this intersection is nonempty. By our choice of $\Sigma$, this means that $\mathcal A$ thinks that $\bigcap_{i<\delta}P_i$ is nonempty whenever each $P_i$ is an element of $U$, which means exactly that $U$ is a ${<}\delta$-complete ultrafilter on $X$.

  \medskip

  For the second statement, assume that $\kappa$ has the $(\delta,\lambda^{<\delta})$-filter extension property. Let $\Sigma$ be a ${<}\kappa$-satisfiable $\mathcal L_{\delta,\omega}$-theory, in a language $\mathcal L$ of size at most $\lambda$. By the regularity of $\delta$, it follows that $\Sigma$ has size at most $\lambda^{<\delta}$. We want to show that $\Sigma$ is satisfiable. 
  %$\mathcal L_{\delta^+,\omega}$-theory. 
  %By introducing Skolem functions, we may assume that~$\Sigma$ contains only universal sentences. 
  For every $\Delta\subseteq\Sigma$ of size less than~$\kappa$, let $\mathfrak A_\Delta\models\Delta$. Let $\mathfrak A=\prod_{\Delta\in[\Sigma]^{<\kappa}}\mathfrak A_\Delta$, and let $A$ denote the domain of $\mathfrak A$. Using a L\"owenheim-Skolem type argument, we may construct a %n $\mathcal L_{\delta,\omega}$-elementary 
  substructure $\mathfrak B$ of $\mathfrak A$ of size at most $\lambda^{<\delta}$, with domain~$B$, and with the additional property that whenever $\varphi(x_1,\dots,x_{n+1})\in\mathcal L_{\delta,\omega}$ and $a=(a_1,\ldots,a_n)\in B^n$ for some $n<\omega$, there is $b\in B$ such that for every $\Delta\in[\Sigma]^{<\kappa}$, if $\mathfrak A_\Delta\models\exists x\,\varphi(x,a(\Delta))$, then $\mathfrak A_\Delta\models\varphi(b(\Delta),a(\Delta))$, where $a(\Delta):=(a_1(\Delta),\ldots,a_n(\Delta))$. This is possible for the number of $\mathcal L_{\delta,\omega}$-formulas is at most~$\lambda^{<\delta}$. For $\varphi\in\Sigma$, let $\hat\varphi=\{\Delta\in[\Sigma]^{<\kappa}\mid\varphi\in\Delta\}$. Obviously, the family $\{\hat\varphi\mid\varphi\in\Sigma\}$ is 
  the base of a ${<}\kappa$-complete filter we shall call $F$. For every %universal
  $\mathcal L_{\delta,\omega}$-formula 
  %$\mathcal L_{\delta^+,\omega}$-formula
  $\varphi(x_1,\ldots,x_n)$ and any $ a\in A^n$, let \[J_{\varphi, a}=\{\Delta\in[\Sigma]^{<\kappa}\mid\mathfrak A_\Delta\models\varphi(a(\Delta))\}.\] Using our assumption, let $U\supseteq F$ be\footnote{Note that $U\supseteq F$ just means that $U$ is a \emph{fine} filter (see also Section \ref{section:more}).} a ${<}\delta$-complete filter on $[\Sigma]^{<\kappa}$ that measures all sets $J_{\varphi,a}$ for $\mathcal L_{\delta,\omega}$-formulas $\varphi(x_1,\ldots,x_n)$ and $a\in B^n$, of which there are $\lambda^{<\delta}$ many. %Here we crucially use that the size of $B$ is at most $\lambda^{{<}\delta}$.
  We define, for $a,b\in B$, that \[a\sim b\iff\{\Delta\in[\Sigma]^{<\kappa}\mid a(\Delta)=b(\Delta)\}\in U,\] and note that a straightforward argument using that $U$ is a filter that measures any relevant set (for all of them correspond to certain $J_{\varphi,a}$ for $a\in B^n)$ shows that for all $n <\omega$, and all $a, b\in B^n$ with $a_i\sim b_i$ whenever $1\leq i \leq n$, it holds that $f(a) \sim f(b)$ whenever $f$ is an $n$-ary function of~$\mathfrak B$. %, and $R(a) \iff R(b)$ whenever $R$ is an $n$-ary relation of $\mathfrak B$. 
    %\todo[inline]{Roman: I'm slightly confused now concerning the claim that $\sim$ is a congruence. Namely, how does one prove that it behaves nicely with relations?\\ 
    %In the traditional proof of compactness using ultraproducts one considers the quotient of $\mathfrak B/_\sim$ with the same domain, but one explicitly defines the relations as follows $$\mathfrak B/_{\sim} \models R([a]_{\sim}) \iff\{ \Delta \in [\Sigma]^{<\kappa} \mid \mathfrak A_\Delta \models R(a(\Delta)) \} \in U.$$
    %The definition of functions agrees with the definition we give here. I believe that the rest of the proof should just go through with this tweaked definition, in particular the induction over formulas is really by definition now. The rest of the induction only relies on the previous induction steps and does not make use of the specific way we define the relations of $\mathfrak D$.}

  For any $a=(a_1,\ldots,a_n)\in B^n$, let $[a]_\sim=([a_1]_\sim,\ldots,[a_n]_\sim)$. We let $\mathfrak D=\mathfrak B/_{\sim}$ be the quotient structure of $\mathfrak B$ by~$\sim$, with domain $D = B/_{\sim}$, and the following operations are well-defined by the above: If $f$ is an $n$-ary function symbol and $a\in B^n$, $f^{\mathfrak D}([a]_{\sim})=[f^{\mathfrak A}(a)]_{\sim}$. If $R$ is an $n$-ary relation symbol, $R^{\mathfrak D}([a]_{\sim})$ if and only if $\{\Delta\in[\Sigma]^{<\kappa}\mid\mathfrak A_\Delta\models R(a(\Delta))\}\in U$.

  \medskip
  
  By induction on formula complexity, using existential quantification as the only quantifier step,\footnote{That is, we replace subformulas of the form $\forall x\varphi$ by $\lnot\exists x\lnot\varphi$.} we show that for every 
  $\mathcal L_{\delta,\omega}$-formula
  %$\mathcal L_{\delta^+,\omega}$-formula 
  $\varphi(x_1,\dots,x_n)$ and $a \in B^n$, \[\mathfrak D\models\varphi ([a]_\sim)\iff J_{\varphi,a}\in U:\] 
  %If along the induction, we encounter a formula $\varphi(y_1,\ldots,y_m,x_1,\ldots,x_n)$, we let $J_{\varphi,\vec a}=\{\Delta\in[\Sigma]^{<\kappa}\mid\mathfrak A_\Delta\models\forall\vec b\,\varphi(\vec b,\vec a(\Delta))\}$, and show that in this case, we have \[\mathfrak D\models\forall\vec b\,\varphi(\vec b,\vec a)\iff J_{\varphi,\vec a}\in U.\]
  This is true for atomic formulas by definition. For the case of conjunctions of size less than $\delta$, we use the ${<}\delta$-completeness of $U$ in a straightforward way. The case of negation is straightforward. Let us argue for the step of existential quantification in detail. Assume first that $\mathfrak D\models\exists y\varphi(y,[a]_\sim)$. This means, by our inductive hypothesis for $\varphi$, that there is $b\in B$ such that the set \[\{\Delta\in[\Sigma]^{<\kappa}\mid\mathfrak A_\Delta\models\varphi(b(\Delta),a(\Delta))\}\in U.\] But this set is obviously contained in $J_{\exists y\varphi,a}$, which is thus also an element of $U$, as desired. If on the other hand, $J_{\exists y\varphi,a}\in U$, by the properties of $\mathfrak B$, we may pick $b\in B$ such that $\mathfrak A_\Delta\models\varphi(b(\Delta),a(\Delta))$ whenever such $b(\Delta)$ exists, and we know this is the case on a set in $U$. This means that $J_{\varphi,b^\frown a}\in U$. By our inductive hypothesis for~$\varphi$, this implies that $\mathfrak D\models\exists y\varphi(y,[a]_\sim)$, as desired.

  We now claim that $\mathfrak D\models\Sigma$. Therefore, pick ${\sigma\in\Sigma}$. Note that for every ${\Delta \in [\Sigma]^{<\kappa}}$, $\sigma \in \Delta$ implies $\mathfrak A_\Delta \models \sigma$, therefore
  $J_{\sigma,\emptyset}\supseteq\hat\sigma \in U$. Thus, by the equivalence that we have verified above, this means that $\mathfrak D \models \sigma$, as desired. 
  %Let $\psi(x_1,\ldots,x_n)$ be obtained from $\sigma$ by (syntactically) removing all universal quantifiers. Then, for each sequence $\vec a$ of length $n$ from $D$, \[\sigma\in\Delta\in[\Sigma]^{<\kappa}\to\mathfrak A_\Delta\models\psi(\vec a(\Delta)),\] that is $\hat\sigma\subseteq J_{\psi,\vec a}\in U$, and therefore $\mathfrak D\models\psi(\vec a)$, which means that $\mathfrak D\models\sigma$.
\end{proof}

The above obviously yields an equivalence in case $\lambda^{<\delta}=\lambda$. 
The following proposition provides the forward direction of Theorem \ref{th:main}, (2).

\begin{proposition}\label{proposition:extensiontocompactness2}
  Whenever $\delta\le\kappa\le\lambda$ are infinite cardinals, and $\kappa$ has the $(\delta,\lambda)$-filter extension property (or the $(\delta,{<}\lambda)$-filter extension property), then $\kappa$-compactness for ${<}\delta$-list-colouring holds for graphs of size at most $\lambda$ (or of size ${<}\lambda$ respectively). In particular, by Proposition \ref{proposition:filterextensionchar}, this is the case if $\kappa$ is $(\delta,\lambda)$-compact.
\end{proposition}
\begin{proof}
  Assume that $\kappa$ has the $(\delta,\lambda)$-filter extension property, let $\cG=\langle G,E\rangle$ be a graph of size at most $\lambda$, and let $L\colon G\to[\delta]^{<\delta}$. Assume that every ${<}\kappa$-size subgraph of $\cG$ has an $L$-colouring. On the set $\delta^G$, we define $F'$ to be the collection of sets
    \begin{equation*}
        A_H = \{ f \in \delta^G \colon f \upharpoonright H \text{ is an $L$-colouring of $\cH=\langle H,E\upharpoonright H^2\rangle$} \},
    \end{equation*}
    where $H$ ranges over all ${\le}2$-element subsets of $G$. Let $I$ consist of all sets in $F'$ together with $B_i^g=\{f \in \delta^G \colon f(g) = i \}$ for $g\in G$ and $i\in L(g)$, as well as the complements of these sets. Let $F$ be the closure under supersets of $F'$ within~$I$. Note that if $(H_i)_{i<\bar\kappa}$ for some $\bar\kappa<\kappa$ is a family of $\leq 2$-element subsets of $G$, with $H:=\bigcup_{i<\bar\kappa}H_i$, then $\bigcap_{i<\bar\kappa}A_{H_i} \supseteq A_H\ne\emptyset$ by our assumption on~$\mathcal G$. This means that $F$ is a ${<}\kappa$-complete filter. By our assumption on $\kappa$, and since $I$ is of size at most $\lambda$, we obtain a ${<}\delta$-complete filter $U\supseteq F$ which measures $I$. For every $g\in G$, there is a unique $i \in L(g)$ so that $B_i^g\in U$ (this uses the ${<}\delta$-completeness of $U$, that $A_{\{g\}}\in F'$, and that $|L(g)|<\delta$), and we denote this unique element $i$ by $i_g$. Observe that by our choice of $F'$, the assignment $h\colon g\mapsto i_g$ is an $L$-colouring of $\mathcal G$.

    The statement regarding the $(\delta,{<}\lambda)$-filter extension property is verified by exactly the same argument.
\end{proof}

\section{Homomorphism-Compactness}\label{section:homomorphismcompactness}

\begin{definition}\label{def:homcompactness}
  \emph{Homomorphism-compactness} for a relational structure $\mathcal B$ is the statement that, given a structure $\mathcal A$ in the same language, if there is a homomorphism ${h\colon\mathcal X\to\mathcal B}$ for every finite substructure $\mathcal X$ of $\mathcal A$, then there exists a homomorphism from $\mathcal A$ to~$\mathcal B$.
\end{definition}

Note that compactness for $3$-colouring is the same as homomorphism-compactness for the complete graph $K_3$. It is well-known and not difficult to see that the classical results of De Bruijn, Erd\H{o}s and L\"auchli mentioned in the introduction of our paper can be adapted to yield the following (this is implicit for example in \cite{ktv}):

\begin{observation}
  Over $\ZF$, the ultrafilter lemma is equivalent to the statement that whenever $\mathcal A$ is a finite relational structure, homomorphism-compactness for $\mathcal A$ holds.
\end{observation}

We generalize Definition \ref{def:homcompactness} as follows:

\begin{definition}\label{def:kappahomcom}
  Let $\kappa\le\lambda$ be infinite cardinals.
  \begin{itemize}
      \item \emph{$(\kappa,\lambda)$-homomorphism-compact\-ness} for a relational structure $\mathcal B$ is the statement that, given a structure $\mathcal A$ of size at most~$\lambda$, in the same language as~$\mathcal B$, if there is a homomorphism $h\colon\mathcal X\to\mathcal B$ for every ${<}\kappa$-size substructure $\mathcal X$ of~$\mathcal A$, then there exists a homomorphism from $\mathcal A$ to~$\mathcal B$.
      \item \emph{$(\kappa,{<}\lambda)$-homomorphism-compactness} is the same statement for structures~$\mathcal A$ of size less than $\lambda$, and we also allow for $\lambda=\Ord$ in this case.
      \item \emph{$\kappa$-homomorphism-compactness} is $(\kappa,{<}\Ord)$-homomorphism-compactness.
  \end{itemize}
\end{definition}

Note that $\kappa$-compactness for $\delta$-list-colouring is a particular instance of the above:

\begin{observation}
  There is a relational structure $\mathcal B$ with domain $\delta$ such that $\kappa$-com\-pactness for $\delta$-list-colouring is equivalent to $\kappa$-homo\-morphism-compactness for~$\mathcal B$.
\end{observation}
\begin{proof}
 Let $\mathcal B$ be the complete graph $K_\delta$ with domain $\delta$ and with unary predicates~$c_x$ for every $x\subseteq\delta$, such that $c_x^\mathcal B(\gamma)$ holds if and only if $\gamma\in x$. If we consider a graph~$\G$ and $L\colon G\to\cP(\delta)$, we expand $\G$ to a structure $\mathcal A$ in the same language as~$\mathcal B$, by letting $c_x^\mathcal A(g)$ hold if and only if $L(g)=x$. Now, a homomorphism from $\mathcal A$ to~$\mathcal B$ corresponds exactly to an $L$-colouring of $\G$. On the other hand, any structure $\mathcal A=\langle G,E,(c_x)_{x\subseteq\delta}\rangle$, where $\G=\langle G,E\rangle$ may be assumed to be a graph and the $c_x$ are unary predicates, naturally yields a map $L\colon G\to\cP(\delta)$, letting $L(g)=\bigcap\{x\mid c_x^\mathcal A(g)\}$, for every $g\in G$. Then, an $L$-colouring of $G$ corresponds exactly to a homomorphism from $\mathcal A$ to $\mathcal B$.
 %The structure $\mathcal C$ is constructed similarly, as the complete graph $K_\delta$ with domain~$\delta$ and with unary predicates $c_x$ for every $x\in[\delta]^{<\delta}$, such that $c_x^\mathcal C(\gamma)$ holds if and only if $\gamma\in x$. An analogous argument as in the first case works to verify the equivalence statement with respect to $\mathcal C$.
 %\todo[inline]{R: $\kappa$-compactness for ${<}\delta$-list-colouring is not an instance of $\kappa$-compactness in general. The problem in translating the argument lies in the following. If $\mathcal A =\langle G,E,(c_x)_{x\in[\delta]^{{<}\delta}}\rangle$ is a structure and $g\in G$ is contained in no relation of the form $c_x^{\mathcal A}$, then $L(g) = \delta$ but this is not allowed for ${<}\delta$-list-colouring.
 %I would suggest to remove the ${<}\delta$-list-coloring from this observation or remark that it is not an instance of homomorphism-compactness. Namely, suppose that $\kappa$ is strongly compact, so in particular $\kappa$-compactness holds for ${<}\kappa$-list-coloring. However, $\kappa$-compactness does not hold for $\mathcal C$.}
 %\todo[inline]{Peter: Ah nice, yes, I guess I didn't check that part carefully enough. On a semi-related note, maybe we should mention somewhere early on in the paper / introduction the trivial fact that $\kappa$-compactness for $\kappa$-list colouring is inconsistent, as witnessed by the complete graph of size $\kappa^+$ (and just allowing arbitrary colours). Like as a small comment after we provide the characterization of strongly compact cardinals via list colouring.}
\end{proof}

We can now provide our promised generalization of Proposition \ref{prop:bm}. The argument is similar to the proof of Proposition \ref{proposition:extensiontocompactness2}. This proposition will also provide the forward directions in Theorem \ref{th:main}, (1) and Theorem \ref{th:main2}.

\begin{proposition}\label{proposition:extensiontocompactness}
  If $\delta\le\kappa\le\lambda$ are infinite cardinals and $\kappa$ has the $(\delta,\lambda)$-filter extension property (or the $(\delta,{<}\lambda)$-filter extension property), then whenever~$\mathcal B$ is a relational structure with domain of size less than $\delta$, $(\kappa,\lambda)$-homomorphism-compactness (or $(\kappa,{<}\lambda)$-homomorphism-compactness) holds for $\mathcal B$. In particular, by Proposition~\ref{proposition:filterextensionchar}, this is the case if $\kappa$ is $(\delta,\lambda)$-compact.
\end{proposition}
\begin{proof}
  Assume that $\kappa$ has the $(\delta,\lambda)$-filter extension property, and, without loss of generality, let $\mathcal B$ be a relational structure with domain $\bar\delta<\delta$. Let $\mathcal A$ be a structure with domain $A$ of size at most $\lambda$, and in the same language as $\mathcal B$. Assume that for every ${<}\kappa$-size substructure $\mathcal X$ of $\mathcal A$, there is a homomorphism from $\mathcal X$ to $\mathcal B$.
  On the set $\bar\delta^A$, we define $F'$ to be the collection of sets
    \begin{equation*}
        A_H = \{ f \in \bar\delta^A \colon f \upharpoonright H \text{ is a homomorphism to $\mathcal B$} \},
    \end{equation*}
    where $H$ ranges over all finite subsets of $A$. Let $E$ consist of all sets in $F'$ together with $B_i^x=\{f \in \bar\delta^A \colon f(x) = i \}$ for $x\in A$ and $i<\bar\delta$, as well as the complements of these sets. Note that $E$ has size at most $\lambda$. Let $F$ be the closure under supersets of $F'$ within $E$. Note that if $(H_i)_{i<\bar\kappa}$ for some $\bar\kappa<\kappa$ is a family of finite subsets of $A$, with $H:=\bigcup_{i<\bar\kappa}H_i$, then $\bigcap_{i<\bar\kappa}A_{H_i} \supseteq A_H\ne\emptyset$ by our assumption on~$\mathcal A$, since $|H|<\kappa$. This means that $F$ is a ${<}\kappa$-complete filter. By our assumption on~$\kappa$, and since $E$ is of size at most $\lambda$, we obtain a ${<}\delta$-complete filter $U\supseteq F$ which measures $E$. For every $x\in A$, there is a unique $i < \bar\delta$ so that $B_i^x\in U$ (this uses the ${<}\delta$-completeness of $U$), and we denote this unique element $i$ by $i_x$. Observe that the assignment $h\colon x\mapsto i_x$ is a homomorphism from $\mathcal A$ to $\mathcal B$: 
    if some relation $R$ holds for the finite tuple $a =(a_1,\dots,a_n)$ in~$\mathcal A$, then 
    %$A_{\vec a}\subseteq\{f\in\bar\delta^A\colon R^{\mathcal B}(f(\vec a))\}\in U$,
    \begin{equation*}
        \bigcap_{1 \leq i \leq n} B^{a_i}_{h(a_i)} \cap  A_a \not = \emptyset,
    \end{equation*}
    which readily implies $R^{\mathcal B}(h(\vec a))$.

    The statement regarding the $(\delta,{<}\lambda)$-filter extension property is verified by exactly the same argument. %, using a structure $\mathcal A$ of size less than $\lambda$ and observing that the set $E$ defined as in the previous case will then also have size less than $\lambda$.
  %But it should also hold at $\delta$-strongly compact cardinals: Let $\mathcal A$ and $\mathcal B$ be relational structures in the same language, with a of size less than $\delta$, with less than $\delta$-many relations $(R_i)_{i<\lambda}$. Let $\mathcal C$ be the structure with domain $A\cup B$, in the same language, and with all the relations just being the (disjoint) unions of the relations on $\mathcal A$ and $\mathcal B$. Now define a structure $\mathcal D$ that expands $\mathcal C$ by adding constant symbols for every element of $A$ and for every element of $B$, as well as a function symbol $h$. Let $\Sigma$ be the $\mathcal L_{\delta,\omega}$-theory that states for every finite tuple $\vec c$ from either $A$ or $B$ exactly which elements of $C$ satisfy which relations $R_i$ as a single $L_{\delta,\omega}$ sentence (using the constant symbols), together with the statements expressing that $h\colon\mathcal A\to\mathcal B$ is a homomorphism. Then, by our assumption, $\Sigma$ is ${<}\kappa$-satisfiable. Let $\mathcal G$ with $G=E\cup F$ be a model of $\Sigma$. The problem now is however that $h$ may map things to $F\setminus B$ (where $B$ means the interpretations of constants for elements of $B$ in $\mathcal G$ here). So how do we finish the argument? Probably not this way at all...
\end{proof}

\section{Some Graph Constructions}\label{section:graphconstructions}

Before we can prove our main results, we need to construct a couple of graphs (and lists of allowed colours) that we will then make use of in our arguments. Impatient readers may skip to Section \ref{section:main}, and refer back to the present section when necessary. The graph $\cR_3$ below already appears in \cite{cowen}.

\begin{lemma}\label{lem:prePreStockmeyer}
    For every $2\leq \alpha < \omega$ there exists a finite graph $\cR_\alpha$ with distinguished vertices $x_0,\dots, x_{\alpha-1}$ and $t$ such that
    \begin{enumerate}
        \item[(i)] if $c$ is a $3$-colouring of $\cR_\alpha$ such that all of $x_0,\dots,x_{\alpha-1}$ receive the same colour, then also $t$ has to receive that colour;
        \item[(ii)] for any $a \in \{0,1\}^\alpha$ and $b \in \{a(i) \mid i < \alpha \}$ there is a 3-colouring $c$ of $\cR_\alpha$ such that $c(x_i) = a(i)$ for $i<\alpha$ and $c(t) = b$. 
    \end{enumerate}
\end{lemma}
\begin{proof}
    We construct $\cR_\alpha = \langle R_\alpha,E \rangle$ by induction over $\alpha$. 
    We let $\cR_2$ be the graph with vertex set $\{x_0,x_1,a,b,t\}$ such that the restriction to $\{a,b,t\}$ is complete and with edges between $x_0$ and $a$ and between $x_1$ and $b$, see Figure~\ref{fig:prePreStockmeyer2}. One easily verifies~(i) and~(ii). 

    For $\alpha\geq2$ we construct $\cR_{\alpha + 1}$ as follows. We take one copy of $\cR_{\alpha}$, whose distinguished vertices we call  $x_0,\dots,x_{\alpha-2},x'_{\alpha-1}$ and $t$, and one copy of $\cR_2$, whose distinguished vertices we call $x_{\alpha-1},x_\alpha$ and $t'$, and let $\cR_{\alpha+1}$ be the quotient graph obtained by identifying $x'_{\alpha-1}$ with $t'$. The construction is depicted in Figure~\ref{fig:prePreStockmeyerSucc}. Using that $\cR_\alpha$ and $\cR_2$ satisfy properties~(i) and~(ii) inductively, one easily concludes that $\cR_{\alpha+1}$ does so as well.
\end{proof}

\begin{figure}[h!]
\centering
\begin{minipage}{0.49\textwidth}
    \centering
    \begin{tikzpicture}[vertex/.style={circle, fill, inner sep=2pt}]
     % Nodes
        \node[vertex,label=left:$a$] (a) at (0,0) {};
        \node[vertex,label=right:$b$] (b) at (2,0) {};
        \node[vertex,label=above:$t$] (t) at (1,1.5) {};
        \node[vertex,label=below:$x_0$] (x0) at (0,-1.5) {};
        \node[vertex,label=below:$x_1$] (x1) at (2,-1.5) {};

    % Complete graph on {a,b,t}
        \draw (a) -- (b);
        \draw (a) -- (t);
        \draw (b) -- (t);

    % edges
        \draw (x0) -- (a);
        \draw (x1) -- (b);
    \end{tikzpicture}
    \caption{The graph $\cR_2$.}        \label{fig:prePreStockmeyer2}
\end{minipage}
    \hfill
    \begin{minipage}{0.49\textwidth}
    \centering
    \begin{tikzpicture}[vertex/.style={circle, fill, inner sep=2pt}, blob/.style={draw, dotted, fill=gray!10}]

% --- First blob S_alpha ---
\draw[blob]
(-1.1,-1) -- (1.1,-1) % flat bottom
.. controls (1.1,0.6) and (.5,0.7) .. (0,0.7)
.. controls (-0.5,0.7) and (-1.1,0.6) .. (-1.1,-1)
-- cycle;

\node[vertex,label=above:$t$] (t) at (0,0.7) {};
\node[vertex,label=below:$x_0$] (x_0) at (-1.1,-1) {};
\node[vertex,label=below:$x_1$] (x_1) at (-0.6,-1) {};
\node[vertex](x_alpha-2) at (0.6,-1) {};
\node[vertex] (x_a) at (1.1,-1) {};
\node at (0,-1.2) {$\cdots$};
\node at (0,0.2) {$\cR_\alpha$};

    % Nodes
        \node[vertex] (a) at (0.6,-1.75) {};
        \node[vertex] (b) at (1.6,-1.75) {};
        \node[vertex,label=right:${t'=x'_{\alpha-1}}$] (t') at (1.1,-1) {};
        \node[vertex,label=below:$x_{\alpha-1}$] (xalpha-1) at (0.6,-2.4) {};
        \node[vertex,label=below:$x_\alpha$] (xalpha) at (1.6,-2.4) {};

    % edges
        \draw (a) -- (b);
        \draw (a) -- (t');
        \draw (b) -- (t');
        \draw (xalpha-1) -- (a);
        \draw (xalpha) -- (b);

    \node[fill=white, fill opacity=0.85, inner sep=2pt, below=2pt of x_alpha-2] {$x_{\alpha-2}$};
    \end{tikzpicture}
    \caption{The graph $\cR_{\alpha+1}$.}
    \label{fig:prePreStockmeyerSucc}
\end{minipage}

\end{figure}

\begin{lemma}\label{lem:preStockmeyer}
    Let $\alpha\geq 2$ be an ordinal. There exists a graph $\cS_\alpha = \langle S_\alpha, E \rangle$, which is of size $|\alpha|$ for infinite $\alpha$, and finite otherwise, with distinguished vertices $y_\alpha$ and $x_i$ for $i < \alpha$, and there exists a function $L_\alpha\colon S_\alpha \to \cP(\alpha+1)$ such that  
    \begin{enumerate}
        \item[(i)] if $c$ is an $L_\alpha$-colouring of $\cS_\alpha$ such that $\{x_i \mid i < \alpha \}$ is monochromatic, then $c(y_\alpha) = \alpha$ and furthermore, there are $L_\alpha$-colourings extending both monochromatic colourings of \{$x_i \mid i < \alpha\}$;
        \item[(ii)] for any non-constant $a \in \{0,1\}^\alpha$ and any $b \in \{0,1,\alpha\}$, there is an $L_\alpha$-colouring $c$ of $\cS_\alpha$ such that $c(x_i) = a(i)$ for $i < \alpha$ and $c(y_\alpha) = b$.
    \end{enumerate}
\end{lemma}
\begin{proof}
    We construct $\cS_\alpha$ and $L_\alpha$ by induction over all ordinals $\alpha \geq 2$. For finite $\alpha \geq 2$, we construct~$\cS_\alpha$ as follows. We take two disjoint copies of $\cR_\alpha$, as obtained by Lemma~\ref{lem:prePreStockmeyer}, and whose distinguished vertices we denote by $x_i$ for $i<\alpha$ and $t$, as well as $x_i'$ for $i<\alpha$ and~$t'$ respectively. For every $i<\alpha$, we add an edge between the vertices $x_i$ and~$x_i'$, and furthermore, we connect both the vertices $t$ and $t'$ with a newly added vertex $y_\alpha$.
    We define $L_\alpha(x_i) = L_\alpha(x_i') = \{0,1\}$ for $i < \alpha$ and let $L_\alpha$ have value $\{0,1,\alpha\}$ on the remaining vertices of $\cS_\alpha$; see Figure~\ref{fig:preStockmeyerFinite} for a graphical representation. Using the properties of $\cR_\alpha$, it is easy to verify that $\cS_\alpha$ satisfies (i) and (ii). 

    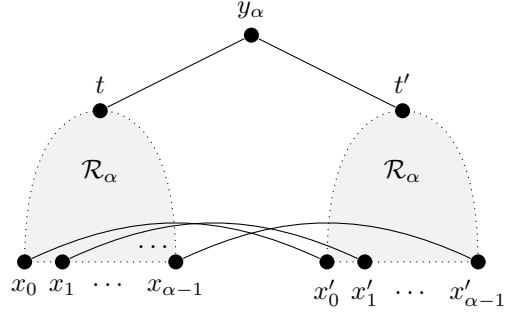
\begin{figure}[t]
    \centering
\begin{tikzpicture}[vertex/.style={circle, fill, inner sep=2pt},
blob/.style={draw, dotted, fill=gray!10}
]

% --- First blob S_alpha ---
\draw[blob]
(-1,-1) -- (1,-1) % flat bottom
.. controls (1,0.5) and (0.5,1) .. (0,1)
.. controls (-0.5,1) and (-1,0.5) .. (-1,-1)
-- cycle;

\node[vertex,label=above:$t$] (t) at (0,1) {};
\node[vertex,label=below:$x_0$] (x_0) at (-1,-1) {};
\node[vertex,label=below:$x_1$] (x_1) at (-0.5,-1) {};
\node[vertex,label=below:$x_{\alpha-1}$] (x_a) at (1,-1) {};
\node at (0.1,-1.3) {$\ldots$};
\node at (0,0.2) {$\cR_\alpha$};

% --- Second blob S_alpha ---
\begin{scope}[xshift=4cm]

\draw[blob]
(-1,-1) -- (1,-1) % flat bottom
.. controls (1,0.5) and (0.5,1) .. (0,1)
.. controls (-0.5,1) and (-1,0.5) .. (-1,-1)
-- cycle;

\node[vertex,label=above:$t'$] (t') at (0,1) {};
\node[vertex,label=below:$x_0'$] (x_0') at (-1,-1) {};
\node[vertex,label=below:$x_1'$] (x_1') at (-0.5,-1) {};
\node[vertex,label=below:$x_{\alpha-1}'$] (x_a') at (1,-1) {};
\node at (0.1,-1.4) {$\ldots$};
\node at (0,0.2) {$\cR_\alpha$};

\end{scope}

% --- Top vertex ---
\node[vertex,label=above:$y_{\alpha}$] (y) at (2,2) {};

\draw[] (y) -- (t);
\draw[] (y) -- (t');
\draw[bend left=25] (x_0) to (x_0');
\draw[bend left=25] (x_1) to (x_1');
\draw[bend left=25] (x_a) to (x_a');
\node at (0.7,-0.8) {$\ldots$};

\end{tikzpicture}
    \caption{The graph $\cS_\alpha$ for finite $\alpha \geq 2$.}
    \label{fig:preStockmeyerFinite}
\end{figure}

    Next, we deal with the case when $\alpha$ is not a cardinal. We have already defined $\cS_{|\alpha|}$, with distinguished vertices $y_{|\alpha|}$ and $x_i$ for $i<|\alpha|$, as well as the function $L_{|\alpha|}$. Pick a bijection $f_\alpha$ from $|\alpha|$ to $\alpha$. We let $\cS_\alpha$ be the same graph as $\cS_{|\alpha|}$. We let $y_\alpha$ be the distinguished vertex $y_{|\alpha|}$ of $\cS_{|\alpha|}$, and we replace the distinguished vertex $x_i$ of $\cS_{|\alpha|}$ by $x_{f_\alpha(i)}$ for every $i<|\alpha|$. For $v\in S_\alpha$, we let $L_\alpha(v)\subseteq \alpha +1$ be the same set as $L_{|\alpha|}(v)$, except that we replace $|\alpha|$ by $\alpha$ whenever $|\alpha| \in L_{|\alpha|}(v)$. The desired properties of $\cS_\alpha$ and~$L_\alpha$ follow immediately from the properties of $\cS_{|\alpha|}$ and $L_{|\alpha|}$.

    Finally, let $\alpha$ be an infinite cardinal. For every $2 \leq \beta < \alpha$, we have a graph $\cS_\beta$, whose distinguished vertices we shall call $y^\beta_{\beta}$ and $x^\beta_i$ (rather than just $y_\beta$ and $x_i$) for $i < \beta$, and a function $L_\beta$, meeting conditions (i) and (ii). We construct $\cS_{\alpha}$ as follows. Let~$\cG$ be a graph with vertices $t$, $t'$, $y_\alpha$, and $x_i, x_i'$ for $i < \alpha$, as well as $y_i, y_i'$ for $2\leq i < \alpha$, and edges from $y_\alpha$ to $t$ and to $t'$, as well as edges from $x_i$ to $x_i'$, from $t$ to $x_i$ and $y_i$ and from $t'$ to both $x_i'$ and $y'_i$ for all possible $i < \alpha$. For every $2 \leq \beta < \alpha$, we glue two copies of the graph $\cS_\beta$ into $\cG$ as follows. For the first copy, to which we will simply refer as $\cS_\beta$ in the following, we identify the vertex
    $y_\beta$ of $\cG$ with the vertex $y^\beta_\beta$ of $\cS_\beta$, and for all $i < \beta$, the vertex $x_i$ of $\cG$ with the vertex $x_i^\beta$ of $\cS_\beta$.
    %and we add an edge between $t$ and $y_\beta$.\peter{This edge was already mentioned above.} 
    Similarly, for the second copy, to which we will refer as $\cS_\beta'$ in the following,
    we identify the vertex $y'_\beta$ of $\cG$ with the vertex $y^\beta_\beta$ of $\cS_\beta'$, and for all $i < \beta$, the vertex $x_i'$ of $\cG$ with the vertex $x_i^\beta$ of $\cS_\beta'$.
    %, and add an edge between $t'$ and $y_\beta$.\peter{Same as previous comment.}  
    We thus the graph $\cS_\alpha$ visualized in Figure~\ref{fig:preStockmeyerLimit}.

\begin{figure}
    \centering
\begin{tikzpicture}[vertex/.style={circle, fill, inner sep=2pt},
blob/.style={draw, dotted, fill=gray!10}
]

% --- First blob (filled) ---
\draw[blob]
(0,0) -- (1,0)
.. controls (1,0.75) and (0.75,1) .. (0.5,1)
.. controls (0.25,1) and (0,0.75) .. (0,0)
-- cycle;

% --- Second blob (on top) ---
\draw[blob]
(0,0) -- (2,0)
.. controls (2,1.5) and (1.5,2) .. (1,2)
.. controls (0.5,2) and (0,1.5) .. (0,0)
-- cycle;

% --- third blob (on top) ---
\draw[blob]
(0,0) -- (4,0)
.. controls (4,3) and (3,4) .. (2,4)
.. controls (1,4) and (0,3) .. (0,0)
-- cycle;

% --- Redraw first outline only (so it stays visible) ---
\draw[dotted]
(0,0) -- (1,0)
.. controls (1,0.75) and (0.75,1) .. (0.5,1)
.. controls (0.25,1) and (0,0.75) .. (0,0)
-- cycle;

\draw[dotted]
(0,0) -- (2,0)
.. controls (2,1.5) and (1.5,2) .. (1,2)
.. controls (0.5,2) and (0,1.5) .. (0,0)
-- cycle;

\node at (2.,2.) {$\cdot$};
\node at (2.1,2.1) {$\cdot$};
\node at (2.2,2.2) {$\cdot$};

\node at (3.7,3.7) {$\cdot$};
\node at (3.8,3.8) {$\cdot$};
\node at (3.9,3.9) {$\cdot$};

% Nodes

\node[vertex,label=below:$x_0$] (x0) at (0,0) {};
\node[vertex,label=below:$x_1$] (x1) at (1,0) {};
\node[vertex,label=below:$x_2$] (x2) at (2,0) {};
\node[vertex,label=below:$x_\beta$] (xbeta) at (4,0) {};

\node[vertex,label=above:\small$y_2$] (y2) at (0.5,1) {};
\node[vertex,label=above:\small$y_3$] (y3) at (1,2) {};
\node at (3,-0.2) {$\cdots$};
\node[vertex,label=above:\small$y_{\beta+1}$] (ybeta) at (2,4) {};
\node at (5.2,-0.2) {$\cdots$ \small$\beta\!<\!\alpha$};
\node[vertex,label=above:$t$] (t) at (4.5,4.5) {};

\draw[bend right=15] (x0) to (t);
\draw[bend right=15] (x1) to (t);
\draw[bend right=15] (x2) to (t);
\draw[bend right=5] (xbeta) to (t);
\draw[bend left=10] (y2) to (t);
\draw[bend left=15] (y3) to (t);
\draw[] (ybeta) -- (t);

\node[] at (0.5,0.6) {\small$\cS_2$};
\node[] at (1,1.2) {\small$\cS_3$};
\node[fill=gray!10, fill opacity=0.85, inner sep=1pt] at (2,2.5) {\small$\cS_{\beta+1}$};

% --- Second blob S_alpha ---
\begin{scope}[xshift = 7cm]

\draw[blob]
(0,0) -- (1,0)
.. controls (1,0.75) and (0.75,1) .. (0.5,1)
.. controls (0.25,1) and (0,0.75) .. (0,0)
-- cycle;

% --- Second blob (on top) ---
\draw[blob]
(0,0) -- (2,0)
.. controls (2,1.5) and (1.5,2) .. (1,2)
.. controls (0.5,2) and (0,1.5) .. (0,0)
-- cycle;

% --- third blob (on top) ---
\draw[blob]
(0,0) -- (4,0)
.. controls (4,3) and (3,4) .. (2,4)
.. controls (1,4) and (0,3) .. (0,0)
-- cycle;

% --- Redraw first outline only (so it stays visible) ---
\draw[dotted]
(0,0) -- (1,0)
.. controls (1,0.75) and (0.75,1) .. (0.5,1)
.. controls (0.25,1) and (0,0.75) .. (0,0)
-- cycle;

\draw[dotted]
(0,0) -- (2,0)
.. controls (2,1.5) and (1.5,2) .. (1,2)
.. controls (0.5,2) and (0,1.5) .. (0,0)
-- cycle;

\node at (2.,2.) {$\cdot$};
\node at (2.1,2.1) {$\cdot$};
\node at (2.2,2.2) {$\cdot$};

\node at (3.7,3.7) {$\cdot$};
\node at (3.8,3.8) {$\cdot$};
\node at (3.9,3.9) {$\cdot$};

% Nodes

\node[vertex,label=below:$x_0'$] (x0') at (0,0) {};
\node[vertex,label=below:$x_1'$] (x1') at (1,0) {};
\node[vertex,label=below:$x_2'$] (x2') at (2,0) {};
\node[vertex,label=below:$x_\beta'$] (xbeta') at (4,0) {};

\node[vertex,label=above:\small$y_2'$] (y2') at (0.5,1) {};
\node[vertex,label=above:\small$y_3'$] (y3') at (1,2) {};
\node at (3,-0.2) {$\cdots$};
\node[vertex,label=above:\small$y_{\beta+1}'$] (ybeta') at (2,4) {};
\node at (5.2,-0.2) {$\cdots$ \small$\beta\!<\!\alpha$};
\node[vertex,label=above:$t'$] (t') at (4.5,4.5) {};

\draw[bend right=15] (x0') to (t');
\draw[bend right=15] (x1') to (t');
\draw[bend right=15] (x2') to (t');
\draw[bend right=5] (xbeta') to (t');
\draw[bend left=10] (y2') to (t');
\draw[bend left=15] (y3') to (t');
\draw[] (ybeta') -- (t');

\node[] at (0.5,0.6) {\small$\cS_2'$};
\node[] at (1,1.2) {\small$\cS_3'$};
\node[fill=gray!10, fill opacity=0.85, inner sep=1pt] at (2,2.5) {\small$\cS_{\beta+1}'$};

\end{scope}

--- Top vertex ---

\node[vertex, label=above:$y_\alpha$] (yalpha) at (8,5.75) {};
\draw[] (yalpha) -- (t);
\draw[] (yalpha) -- (t');

\end{tikzpicture}
    \caption{The graph $\cS_\alpha$ for an infinite cardinal $\alpha$. For better readability the edges between $x_i$ and $x_i'$ are omitted.}
    \label{fig:preStockmeyerLimit}
\end{figure}
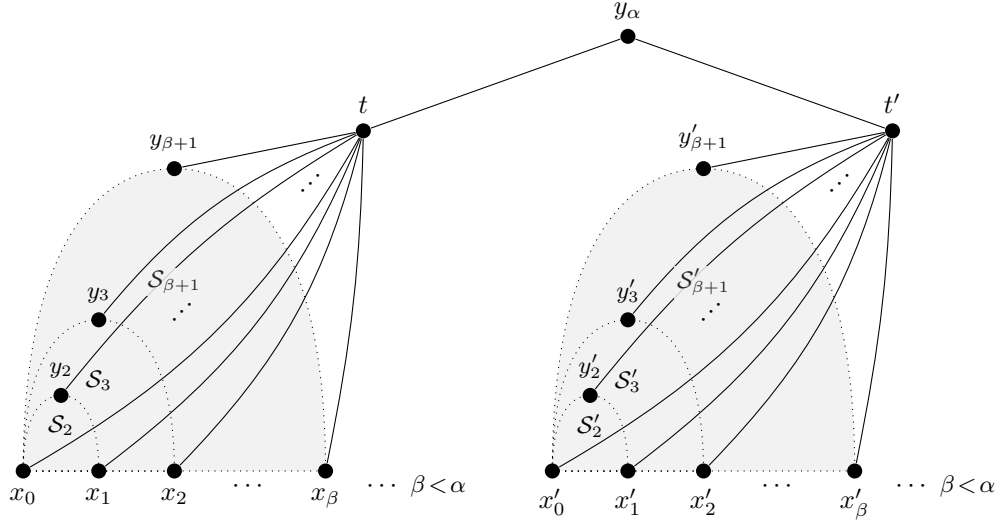

    The function $L_\alpha \colon S_\alpha \to \cP(\alpha + 1)$ is defined as follows. If $x \in S_\alpha$ is an element of a glued-in copy of $\cS_\beta$ for some $2 \leq \beta < \alpha$, we let $L_\alpha(x) = L_\beta(x)$. Note that this assignment is well-defined, as $x$ being contained in multiple glued-in copies of $\cS_\beta$ implies that it equals $x_i$ or $x_i'$ for some $i < \alpha$, in which case $L_\beta(x)=\{0,1\}$ for all $\beta > i$. Furthermore, we let $L_\alpha(t) = L_\alpha(t') = \alpha$ and $L_\alpha(y_\alpha) = \{0,1,\alpha\}$.
    
    $\cS_\alpha$ and $L_\alpha$ satisfy (i) and~(ii): To see that (i) holds, suppose that $c$ is an $L_\alpha$-colouring of $\cS_\alpha$ such that $\{x_i \mid i < \alpha \}$ is monochromatic, say $c(x_i) = 0$ for all $i<\alpha$. 
    Since for every $2 \leq \beta < \alpha$, $L_\alpha \upharpoonright {\cS_\beta} = L_\beta$, we conclude that $c(y_\beta) = \beta$ by property (i) of~$\cS_\beta$ and $L_\beta$. The edges from $x_i$ to $x_i'$ for all $i < \alpha$ imply that $c(x_i') = 1$ for all $i < \alpha$, so the analogous argument shows that $c(y_\beta') = \beta$ for all $2 \leq \beta < \alpha$.
    Since $t$ has edges to $x_0$ as well as to $y_\beta$ for all $\beta < \alpha$, we have $c(t) = 1$ and similarly $c(t') = 0$. Thus, $c(y_\alpha) = \alpha$, as desired. The above also yields that both monochromatic colourings of $\{x_i \mid i < \alpha \}$ extend to $L_\alpha$-colourings of $\cS_\alpha$.
    
    To see that~(ii) holds, let $a \in \{0,1\}^\alpha$ be non-constant, and let $b \in \{0,1,\alpha\}$. Let $a' \in \{0,1\}^\alpha$ be such that $a(i) =  0 \iff a'(i) = 1$, for all $i < \alpha$.
    Then, since $\alpha$ is a limit ordinal and $a$ is non-constant, there is some $\beta<\alpha$ such that $a \upharpoonright {\beta}$ (and hence also $a'\upharpoonright \beta$) is non-constant. Using our inductive assumption, we find, for every $2 \leq \gamma < \alpha$, $L_\gamma$-colourings~$c_\gamma$ of $\cS_\gamma$ and~$c'_\gamma$ of $\cS'_\gamma$ such that for all $i<\gamma<\alpha$, it holds that $c_\gamma(x_i) = a(i)$, $c_\gamma'(x_i') = a'(i)$ and $c_\gamma(y_\gamma), c_\gamma'(y'_\gamma)<\beta$. We let $c$ be the extension of $\bigcup_{\gamma < \alpha} (c_\gamma\cup c_\gamma')$ to domain $S_\alpha$ with $c(t) = c(t') = \beta$ and $c(y_\alpha) = b$. This is an $L_\alpha$-colouring of $\cS_\alpha$ with %$c(x_i) = a(i)$ for $i<\alpha$ and $c(y_\alpha) = b$
    the desired properties, concluding our induction.
\end{proof}

\begin{lemma}\label{lem:Stockmeyer}
    Let $\delta \geq 3$ be a cardinal. There exists a graph $\cT_\delta = \langle T_\delta,E\rangle$, which is of size $\delta$ for infinite $\delta$, and finite otherwise, with distinguished vertices $y$ and $x_i$ for $i < \delta$, and a map $L'\colon T_\delta \to \cP(\delta)$ with $L'(x_i) = L'(y) = \{0,1\}$ for all $i < \delta$ such that the following holds.
    \begin{enumerate}
        \item[(i)] If $c$ is an $L'$-colouring of $\cT_\delta$ such that $c(x_i) = c(x_j)$ for all $i,j < \delta$, then $c(y) = c(x_0)$.
        \item[(ii)] For every $a \in \{0,1\}^\delta$ and $b \in \{a(i) \mid i < \delta\}$ there is an $L'$-colouring $c\colon T_\delta \to \delta$ such that $c(y) = b$ and $c(x_i) = a(i)$ for all $i < \delta$.  
    \end{enumerate}
\end{lemma}
\begin{proof}
    For finite $\delta \geq 3$ this follows immediately from Lemma~\ref{lem:prePreStockmeyer}, so we assume that $\delta$ is infinite.
    We let $\cG$ be the graph with vertices $y, t, x_\alpha$ for $\alpha<\delta$, and $y_\alpha$ for $2 \leq \alpha < \delta$, with edges between $y$ and $t$, between $x_0$ and $t$, and between $y_\alpha$ and $t$ for $2\leq \alpha< \delta$. Similar to the argument of the previous lemma, for every $2 \leq \alpha < \delta$, we glue a copy of~$\cS_\alpha$, whose distinguished vertices we shall call $Y_\alpha$ and $X^\alpha_i$ for $i<\alpha$, into $\cG$, by identifying the vertex $Y_\alpha$ with $y_\alpha$ and $X^\alpha_i$ with $x_i$ for $i<\alpha$. We obtain the graph $\cT_\delta=\langle T_\delta,E \rangle$ depicted in Figure~\ref{fig:Stockmeyer}. We define $L'\colon T_\delta \to \cP(\delta)$ as follows. We require that on every copy of $\cS_\alpha$ for $2 \leq \alpha < \delta$, $L'$ restricts to the function $L_\alpha$ from Lemma~\ref{lem:preStockmeyer}. Furthermore, we define $L'(t) = \delta$ and $L'(y) = \{0,1\}$. Note that this specifies a well-defined map $T_\delta \to \cP(\delta)$, and one easily verifies properties~(i) and~(ii).
\end{proof}

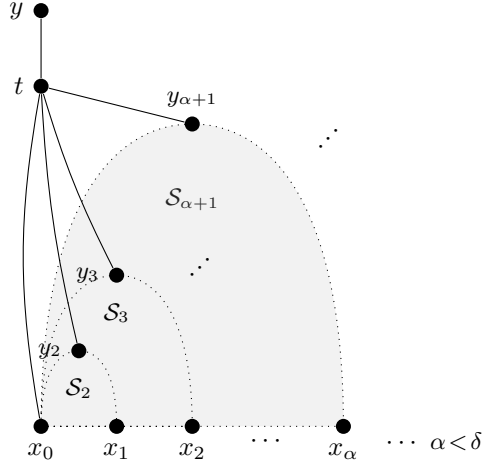
\begin{figure}[t]
    \centering
\begin{tikzpicture}[vertex/.style={circle, fill, inner sep=2pt},
blob/.style={draw, dotted, fill=gray!10}
]

% --- First blob (filled) ---
\draw[blob]
(0,0) -- (1,0)
.. controls (1,0.75) and (0.75,1) .. (0.5,1)
.. controls (0.25,1) and (0,0.75) .. (0,0)
-- cycle;

% --- Second blob (on top) ---
\draw[blob]
(0,0) -- (2,0)
.. controls (2,1.5) and (1.5,2) .. (1,2)
.. controls (0.5,2) and (0,1.5) .. (0,0)
-- cycle;

% --- third blob (on top) ---
\draw[blob]
(0,0) -- (4,0)
.. controls (4,3) and (3,4) .. (2,4)
.. controls (1,4) and (0,3) .. (0,0)
-- cycle;

% --- Redraw first outline only (so it stays visible) ---
\draw[dotted]
(0,0) -- (1,0)
.. controls (1,0.75) and (0.75,1) .. (0.5,1)
.. controls (0.25,1) and (0,0.75) .. (0,0)
-- cycle;

\draw[dotted]
(0,0) -- (2,0)
.. controls (2,1.5) and (1.5,2) .. (1,2)
.. controls (0.5,2) and (0,1.5) .. (0,0)
-- cycle;

\node at (2.,2.) {$\cdot$};
\node at (2.1,2.1) {$\cdot$};
\node at (2.2,2.2) {$\cdot$};

\node at (3.7,3.7) {$\cdot$};
\node at (3.8,3.8) {$\cdot$};
\node at (3.9,3.9) {$\cdot$};

% Nodes

\node[vertex,label=below:$x_0$] (x0) at (0,0) {};
\node[vertex,label=below:$x_1$] (x1) at (1,0) {};
\node[vertex,label=below:$x_2$] (x2) at (2,0) {};
\node[vertex,label=below:$x_\alpha$] (xbeta) at (4,0) {};

\node[vertex,label=left:\small$y_2$] (y2) at (0.5,1) {};
\node[vertex,label=left:\small$y_3$] (y3) at (1,2) {};
\node at (3,-0.2) {$\cdots$};
\node[vertex,label=above:\small$y_{\alpha+1}$] (ybeta) at (2,4) {};
\node at (5.2,-0.2) {$\cdots$ \small$\alpha\!<\!\delta$};
\node[vertex,label=left:$t$] (t) at (0,4.5) {};
\node[vertex,label=left:$y$] (y) at (0,5.5) {};

\draw[bend left = 10] (x0) to (t);
\draw[bend left = 5] (y2) to (t);
\draw[bend left = 5] (y3) to (t);
\draw[] (ybeta) -- (t);
\draw[] (t) -- (y);

\node[] at (0.5,0.5) {\small$\cS_2$};
\node[] at (1,1.5) {\small$\cS_3$};
\node[fill=gray!10, fill opacity=0.85, inner sep=1pt] at (2,3) {\small$\cS_{\alpha+1}$};

\end{tikzpicture}
    \caption{The graph $\cT_\delta$ for a cardinal $\delta$.}
    \label{fig:Stockmeyer}
\end{figure}

\section{Retrieving large cardinals}\label{section:main}

In this section, we provide reversals to Propositions \ref{proposition:extensiontocompactness2} and \ref{proposition:extensiontocompactness}, which, together with Proposition \ref{proposition:filterextensionchar}, establish Theorem \ref{th:main}. Some of the basic ideas for this argument stem from the presentation of L\"auchli's theorem in \cite{cowen}.

\begin{theorem}\label{th:compactnesstofilters}
  Assume that $\delta\le\kappa\le\lambda$ are infinite cardinals, with $\delta$ regular. Then:
  \begin{itemize}
      \item If $\delta<\kappa$, then $\kappa$-compactness for $\delta$-list-colouring for graphs of size at most $\lambda^\delta$ implies that $\kappa$ has the $(\delta^+,\lambda)$-filter extension property.
      \item If $\delta\le\kappa$, then $\kappa$-compactness for ${<}\delta$-list-colouring for graphs of size at most $\lambda^{<\delta}$ implies that $\kappa$ has the $(\delta,\lambda)$-filter extension property.
  \end{itemize}  
\end{theorem}
\begin{proof}
    Assume first that $\kappa$-compactness for $\delta$-list-colouring holds for graphs of size at most $\lambda^\delta$. We want to show that $\kappa$ has the $(\delta^+,\lambda)$-filter extension property.
    Given a set $J$, and $I\subseteq\cP(J)$ of size at most $\lambda$, we may enlarge $I$ to size at most~$\lambda^\delta$, so that $I$ is closed under $\delta$-intersections and complements, and $J\in I$. Let $F\subseteq I$ be a ${<}\kappa$-complete filter. We need to extend $F$ to a ${<}\delta^+$-complete filter $F'\supseteq F$ that measures $I$.  We construct a graph $\cG=\langle G,E\rangle$ with $G \supseteq I$, and a map $L\colon G \to \cP(\delta)$, with the property that every $L$-colouring~$c$ of~$\cG$ corresponds to a ${<}\delta^+$-complete filter that extends $F$ and measures $I$, and which also has the property that $L$-colourings of subgraphs $\cH=\langle H,E\cap H^2\rangle$ of $\cG$ correspond to ${<}\delta^+$-complete filters extending~$F$ and measuring all elements of $I$ in $H$.
    %every subgraph $\cH$ of $\cG$ is contained in a subgraph $\cH' \subseteq \cG$  of size $|H| + \delta$ such that the $L$-colourings of $\cH'$ correspond to the ${<}\delta^+$-complete filters $U\supseteq F\cap H'$ on $I$ that measure all elements of $H'\cap\mathcal P(I)$. In particular, if $H=G$, then $U$ is a ${<}\delta^+$-complete ultrafilter on $I$.
    Let $\cG' = \langle I,E'\rangle$ be the graph with edges $E'=\{(X,J\setminus X)\mid X\in I\}$. We construct $\cG=\langle G,E\rangle$ with $G\supseteq I$ and $E\supseteq E'$ from $\cG'$ by gluing in a copy of $\cT_\delta$ for every family $(X_i)_{i < \delta}$ of elements of $I$, identifying the vertex $x_i$ of $\cT_\delta$ with the vertex $X_i$ of $\cG'$ for all $i<\delta$, and identifying~$y$ with $\bigcap_{i<\delta}X_i\in I$. %\footnote{More formally, we take the disjoint union of $\cG'$ and a copy of $\cT_\delta$ for every family $(X_i)_{i < \delta}$ of elements of $I$, and then take the quotient graph identifying the vertex $x_i$ in the copy of $\cT_\delta$ corresponding to the family $(X_i)_{i < \delta}$ with the vertex $X_i$ of $\cG'$ for all $i < \delta$, as well as identifying $y$ with $\bigcap_{i < \delta} X_i$ to obtain $\cG$.} 
    Note that $G$ has size at most $\lambda^\delta$. We define the function $L\colon G \to \cP(\delta)$ as follows. For $X \in F$, we set $L(X) = \{1\}$, for $X \in I\setminus F$ we set $L(X) = \{0,1\}$ and for $g \in G \setminus I$ we set $L(g)=L'(g)$, where $L'$ refers to the map $T_\delta \to \cP(\delta)$ defined for the respective copy of $\cT_\delta$ from Lemma~\ref{lem:Stockmeyer} that $g$ belongs to.
    Observe that for every $L$-colouring $c$ of $\cG$, the set
    \begin{equation*}
        F' = \{ X \in I \mid c(X) = 1\}
    \end{equation*} 
    is a ${<}\delta^+$-complete filter that extends $F$ and measures $I$: $F'$ extends $F$ by our demand that $L(X)=\{1\}$ whenever $X\in F$, and for every $X\in I$, precisely one of $X$ and $J \setminus X$ is an element of $F'$, for they are connected by an edge in $\cG$. Furthermore, if $(X_i)_{i < \delta}$ is a family of elements of $F'$, i.e., $c(X_i)=1$ for all $i < \delta$, then Property (i) from Lemma~\ref{lem:Stockmeyer} for the suitable copy of $\cT_\delta$ contained in $\cG$ yields $c(\bigcap_{i<\delta}X_i) =1$, so $\bigcap_{i<\delta}X_i \in F'$. Finally, if $X\subseteq Y$ are both in $I$ and $X\in F'$, i.e., $c(X)=1$, then if $c(Y)=0$, it follows by the above that $c(J\setminus Y)=1$, and thus again by the above that $c(X\cap(J\setminus Y))=c(\emptyset)=1$, and hence that $c(J)=0$, contradicting that $J\in F$.

    \medskip

    Thus, it remains to show that there exists an $L$-colouring of $\cG$, which by the $\kappa$-compactness for $\delta$-list-colouring for graphs of size at most $\lambda^\delta$ is equivalent to showing that every subgraph $\cH \subseteq \cG$ of size ${<}\kappa$ has an $L$-colouring. Let $\cH=\langle H,D\rangle$ be a subgraph of $\cG$ of size less than~$\kappa$. We may assume that whenever $h\in H$ is an element of some copy of $\cT_\delta$ in $\cG$ with $h\not\in I$, this copy of $\cT_\delta$ (including its nodes from $I$) is completely contained in $\cH$. The ${<}\kappa$-completeness of $F$ implies that the set $\bigcap (H \cap F)\subseteq J$ is non-empty, so it contains some $a \in J$ as an element. We define an $L$-colouring $c\colon H \to \delta$ as follows. For $X \in H \cap I$, we set $c(X) = 1$ if $a \in X$, and $c(X) = 0$ otherwise. 
    The remaining vertices of $H$ are contained in (unique) copies of $\cT_\delta$, and we define the function $c$ on those vertices for each copy of $\cT_\delta$ as follows. %We distinguish two cases. In the first case, we consider a copy of $\cT_\delta$ such that all of its distinguished vertices $x_i$ for $i<\delta$, and $y$ (which were identified with sets $X_i \in I$, $i<\delta$ and $\bigcap_{i < \delta} X_i$ respectively), are contained in $H$. 
    Let us consider a particular copy of $\cT_\delta$ in $\cH$. Observe that $c(y)$ and $c(x_i)$ for $i<\delta$ have been defined already, and that $c(y) \in \{ c(x_i) \mid i < \delta \}$. Property (ii) from Lemma~\ref{lem:Stockmeyer} %, in combination with the fact that on the copy of $\cT_\delta$ in question, $L$ agrees with the function $L'$ of Lemma~\ref{lem:Stockmeyer}, 
    shows that we find an $L$-colouring of the whole copy of $\cT_\delta$ that extends $c$ as already defined on $H \cap I$. %In the second case, we consider a copy of $\cT_\delta$ such that some of its distinguished vertices are not contained in $H$. It is then easy to see that property~(2) of Lemma~\ref{lem:Stockmeyer} implies that the whole copy of $\cT_\delta$ has an $L$-colouring extending $c$ as defined on $H \cap \cP(I)$, which we may again restrict to the vertices contained in $H$. 
    Since different copies of $\cT_\delta$ in $\cH$ are not connected on nodes outside of $I$, $c$ is an $L$-colouring of~$\cH$.

    \medskip

    Now let us consider the case when $\kappa$-compactness for ${<}\delta$-list-colouring holds for graphs of size at most $\lambda^{<\delta}$. We now want to show that $\kappa$ has the $(\delta,\lambda)$-filter extension property. The argument is basically the same as in the previous case, but now we enlarge a given set $I\subseteq\cP(J)$ of size at most $\lambda$ to size at most $\lambda^{<\delta}$, with $J\in I$, and $I$ closed under ${<}\delta$-intersections and complements. We then construct a graph $\cG$ such that every $L$-colouring of $\cG$ corresponds to a ${<}\delta$-complete filter on $J$ that extends a given ${<}\kappa$-complete filter $F\subseteq I$ on $J$ that measures $I$. 
    The graph is constructed as before, but we now consider families $(X_i)_{i<\bar\delta}$ of elements of $I$ for cardinals $3 \leq \bar\delta<\delta$, and glue a copy of $\cT_{\bar\delta}$ into $\cG$ for every such family. We also define $L$ as before, and note that $L\colon G\to[\delta]^{<\delta}$, since this is the case for every map $L'$ that is associated to $\cT_{\bar\delta}$ for $3 \leq \bar\delta<\delta$. We then define $F'$ as in the previous case, and the same argument, using the $\cT_{\bar\delta}$'s rather than $\cT_\delta$, shows that $F'\supseteq F$ is a ${<}\delta$-complete filter that measures $I$. Finally, the argument that $\cG$ has an $L$-colouring proceeds exactly as in the previous case.
\end{proof}

We also obtain an analogous characterization of certain compact cardinals in terms of homomorphism-compactness:

\begin{theorem}\label{th:main2}
  Assume that $\delta<\kappa\le\lambda$ are infinite cardinals, with $\delta$ regular. Then:
  \begin{itemize}
      \item If $\lambda^\delta=\lambda$, then $\kappa$ is $(\delta^+,\lambda)$-compact if and only if $(\kappa,\lambda)$-homomorphism-compactness holds for target structures of size at most $\delta$.
    \item If $\bar\lambda^\delta<\lambda$ whenever $\bar\lambda<\lambda$, then $\kappa$ is $(\delta^+,{<}\lambda)$-compact if and only if $(\kappa,{<}\lambda)$-homomorphism-compactness holds for target structures of size at most $\delta$.
    \item $\kappa$ is $\delta^+$-strongly compact if and only if $\kappa$-homo\-morphism-compact\-ness holds for target structures of size at most~$\delta$.
  \end{itemize}
\end{theorem}
\begin{proof}
  Immediate by Propositions \ref{proposition:extensiontocompactness}, \ref{proposition:filterextensionchar} and Theorem \ref{th:compactnesstofilters}.
\end{proof}

\section{On the largeness of compact cardinals}\label{section:compact}

The goal of this section is to verify Theorem \ref{theorem:compactlarge}, which states that for any cardinals $\delta\le\kappa$ with $\delta$ regular and uncountable, if $\kappa$ is $(\delta,\kappa)$-compact, then $\kappa\ge\aleph_\delta$.

\begin{proof}[Proof of Theorem~\ref{theorem:compactlarge}]
  Assume for a contradiction that $\kappa$ is a cardinal in $[\delta,\aleph_\delta)$ that is $(\delta,\kappa)$-compact. Let $\mathcal L=\{\alpha_i\mid i\leq \kappa \}\cup\{E,\in,f\}$, 
  with constant symbols $\alpha_i$, a unary predicate $E$ and binary relations $\in$ and $f$.
  For every limit cardinal $\lambda\le\kappa$, we fix a sequence of cardinals $\langle \lambda_i \mid i < \cof(\lambda) \rangle$ below $\lambda$ that is cofinal in $\lambda$.
  %In case $\lambda = \omega$ we choose $\langle i \mid i < \omega \rangle$.
  Let $T$ be the theory consisting of the following:
    \begin{enumerate}
        \item\label{one} for every $i\le\kappa$, $E(\alpha_i)$,
      \item \emph{$E$ is transitive:} $\forall x,y\ [(x\in y\,\land\,E(y))\to E(x)]$,
      \item\label{lo} \emph{$\in$ is a strict linear ordering of all elements satisfying $E$}, that is:
      \begin{itemize}
        \item $\forall x,y,z\ E(z)\to[(x\in y\,\land\,y\in z)\to x\in z]$,
        \item $\forall x\ne y\ [(E(x)\,\land\,E(y))\to(x\in y\,\lor\,y\in x)]$,
        \item $\forall x\ E(x)\to\lnot(x\in x)$,
      \end{itemize}
      \item\label{normal} for all $i<j\le\kappa$, $\alpha_i \in \alpha_j$,
      \item\label{finite} for $m<\omega$, \[\forall x\ \left(x\in\alpha_m\iff\bigvee_{k<m}\ x=\alpha_k\right),\]
      \item\label{limit} for every limit cardinal $\lambda\le\kappa$, \[\forall x\ \left(x \in \alpha_\lambda \iff \bigvee_{i < \cof(\lambda)} x \in \alpha_{\lambda_i}\right),\] 
      \item\label{small} for every cardinal $\lambda\le\kappa$, the sentence that for all $x\in\alpha_\lambda$, $x$ is a surjective image of $\alpha_{\bar\lambda}$ for some cardinal $\bar\lambda < \lambda$,\footnote{Note that since $\lambda<\aleph_\delta$, there are less than $\delta$-many cardinals below $\lambda$, so this can be formulated as an $\mathcal L_{\delta,\omega}$-sentence.} where $z$ being a surjective image of $\alpha_{\bar\lambda}$ means that there is $g$ consisting (via $\in$) of ordered pairs $(x,y)$ with $x\in\alpha_{\bar\lambda}$ and $y\in z$, such that
     \begin{itemize}
       \item $\forall x\in\alpha_{\bar\lambda}\,\exists!y\in z\ (x,y)\in g$, and
      \item $\forall y\, [y\in z\to\exists x\in\alpha_{\bar\lambda}\ (x,y)\in g]$. 
     \end{itemize}
      \item\label{fbasic} $f$ is a function, and its image of $\alpha_\kappa$ contains $\{\alpha_i\mid i\le\kappa\}$, or more precisely:
      \begin{itemize}
        \item $\forall x\exists!y f(x,y)$,
        \item for every $i\le\kappa$, $\exists x\in\alpha_\kappa\ f(x,\alpha_i)$.
      \end{itemize}
      We will write $f(x)=y$ rather than $f(x,y)$ in the following.
      \item\label{fpreserve} $f$ is weakly order-preserving with respect to $\in$, that is if $x\in y \in \alpha_\kappa$, then either $f(x)\in f(y)$ or $f(x)=f(y)$.
    \end{enumerate}
    
  % \begin{enumerate}
  %     \item\label{one} for every $i<\kappa$, $E(\alpha_i)$,
  %     \item \emph{$E$ is transitive:} $\forall x,y\ [(x\in y\,\land\,E(y))\to E(x)]$,
  %     \item\label{lo} \emph{$\in$ is a strict linear ordering of all elements satisfying $E$}, that is:
  %     \begin{itemize}
  %       \item $\forall x,y,z\ E(z)\to[(x\in y\,\land\,y\in z)\to x\in z]$,
  %       \item $\forall x\ne y\ [(E(x)\,\land\,E(y))\to(x\in y\,\lor\,y\in x)]$,
  %       \item $\forall x\ E(x)\to\lnot(x\in x)$,
  %     \end{itemize}
  %     
  %     %\item\label{omega} \emph{$\alpha_\omega$ resembles $\omega$, in the sense that}\[\forall x\ \left(x\in\alpha_\omega\iff\bigvee_{n<\omega}x=\alpha_n\right),\]
  %     \item \todo[inline]{What do we do if $\bar\kappa$ is a limit ordinal?}
  %     \item If $\lambda<\bar\kappa$ is a limit ordinal, we pick a cofinal sequence $\langle\lambda_n\mid n<\omega\rangle$ for $\lambda$, and ask that $\alpha_\lambda$ resembles $\lambda$, in the sense that \[\forall x\ \left(x\in\alpha_\lambda\iff\bigvee_{n<\omega}x\in\alpha_{\lambda_n}\right),\]
  %     \item\label{normal} for every $i<j\le\bar\kappa$, $\alpha_i\in\alpha_j$,
  %     \item\label{step} for every cardinal $\bar{\bar\kappa}\le\bar\kappa$, every element of $\alpha_{\bar{\bar\kappa}}$ is a surjective image of $\alpha_j$ for some $j<\bar{\bar\kappa}$, and
  %     \item\label{small} every $\alpha_i$ is a surjective image of $\alpha_j$ for some $j<\bar\kappa$,
  % \end{enumerate}

  $T$ is an $\mathcal L_{\delta,\omega}$-theory in a language of size $\kappa$.
  
  \begin{claim}
    $T$ is ${<}\kappa$-satisfiable.
  \end{claim}
%    Let $\bar T\subseteq T$ be of size less than $\kappa$. We may assume that $\bar T$ is infinite. Let $R$ be the set (of size ${<}\kappa$) of all $i<\kappa$ for which $\alpha_i$ is mentioned in a sentence in $\bar T$, and assume without loss of generality that $\alpha_\kappa$ is mentioned in a sentence in $\bar T$. We obtain $M\models\bar T$ by letting $\alpha_i=i$ for $i<\omega$ in $R$, and letting $\alpha_i=2\cdot i$ whenever $i\in R$ with $i\ge\omega$. We include elements $\gamma_j=2\cdot j+1$ for $\omega\le j<|\bar T|^+$ into our model~$M$, and let $E$ hold exactly for all $\alpha_i$'s ($i\in R$) and $\gamma_j$'s ($j<|\bar T|^+$). We pick $\in$ to be the actual elementhood relation, and also include suitable surjections (together with all ordered pairs contained in them) into~$M$ in order to satisfy Item~\ref{small}. Let $B:=\{x\in M\mid M\models x\in\alpha_\kappa\}=\{\alpha_i\mid i\in R\cap\kappa\}\cup\{\gamma_j\mid j<|\bar T|^+\}$, and let $A$ be the proper initial segment of $B$ that is order isomorphic (with respect to $\in$) to $\{\alpha_i\mid i\in R\}$, which exists since $|\bar T|^+>|\bar T|\ge|R|$. Let $f$ be the map from $B$ to $\{\alpha_i\mid i \in R\}$ that is the (unique) order isomorphism between $A$ and $\{\alpha_i\mid i\in R\}$ when restricted to $A$, and which maps all elements of $B\setminus A$ to $\alpha_\kappa$. Then, $f$ is weakly order-preserving with respect to $\in$, and it is straightforward to observe that $\bar T$ is satisfied in $M$.
\begin{proof}
    Let $\bar T\subseteq T$ be of size less than $\kappa$. Let $R$ be the set (of size ${<}\kappa$) of all $i\le\kappa$ for which $\alpha_i$ is mentioned in a sentence in $\bar T$. We let $\theta < \kappa$ be the order type of $R$ and write $R = \{ i_\xi \mid \xi \in \theta \}$, with the $i_\xi$ strictly increasing.
    We obtain $M\models\bar T$ as follows. We let $\alpha_i=i$ for $i\leq \kappa$, and let $E$ hold exactly for all $\alpha_i$, $i \leq \kappa$. We pick $\in$ to be the actual elementhood relation, and also include suitable surjections (together with all ordered pairs $(x,y)$, coded as $\{\{x\},\{x,y\}\}$, contained in them, as well as the corresponding sets $\{x\}$ and $\{x,y\}$) into~$M$ in order to satisfy Item~\ref{small}. Trivially, this model satisfies the axioms from items~\ref{one}--\ref{small}; to satisfy the axioms contained in $\bar T$ from the remaining items, we let $f$ be the assignment $\xi \mapsto i_\xi$ for $\xi < \theta$ and $\xi \mapsto \kappa$, for $\theta \leq \xi < \kappa$. For the remaining $x\in M$ we set $f(x) = 0$.
    Then, $f$ is weakly order-preserving with respect to $\in$ for elements of $\alpha_\kappa = \kappa$, as required by Item~\ref{fpreserve}, and it is straightforward to observe that the axioms from Item~\ref{fbasic} that are contained in $\bar T$ are satisfied.
  \end{proof}
    
   We will now show that $T$ is not satisfiable, thus yielding the statement of our theorem. Assume for a contradiction that $M\models T$.
   
  \begin{claim}
    \begin{itemize}
      \item For all cardinals $\lambda\le\kappa$,
    $|\{x\in M\mid M\models x\in\alpha_{\lambda}\}|=\lambda$.
      \item If $\lambda\le\kappa$ is an infinite cardinal, then $\{\alpha_i\mid i<\lambda\}$ is unbounded in $\alpha_\lambda$ with respect to $\in$.
  \end{itemize}
    \end{claim}
  \begin{proof}
    By induction on~$\lambda\leq\kappa$: The start of the induction (for all $\lambda<\omega$) is provided by Item \ref{finite}. If $\lambda\leq\kappa$ is a limit cardinal, the claim is immediate by Item~\ref{limit}. Assume that the claimed statement holds for $\lambda<\kappa$. By Item \ref{normal}, it holds that $|\{x\in M\mid M\models x\in\alpha_{\lambda^+}\}|\ge\lambda^+$.  %By Items \ref{one}--\ref{lo}, these $x$ are linearly ordered by $\in$ in $M$.
    Note that our induction hypothesis implies that there cannot be $x\in M$ with $\forall i<\lambda^+\ M\models\alpha_i\in x\in\alpha_{\lambda^+}$, since the existence of such $x$ would clearly contradict Item \ref{small}. This means that the $\alpha_i$ for $i<\lambda^+$ are unbounded in~$\alpha_{\lambda^+}$ with respect to $\in$ in $M$. This in turn means that \[\{x\in M\mid M\models x\in\alpha_{\lambda^+}\}=\bigcup_{i<\lambda^+}\{x\in M\mid M\models x\in\alpha_i\}.\] By Item \ref{small} with respect to the $\alpha_i$'s, and by our inductive hypothesis, this set has cardinality at most 
    ~$\lambda^+$.
  \end{proof}

  We will now obtain a contradiction by considering the interpretation of $f$ in $M$. By Item \ref{fbasic}, there has to be some $x\in\alpha_\kappa$ such that $f(x)=\alpha_\kappa$. But note that by Item \ref{small} and the previous claim, $\{y\in M\mid M\models y\in x\}$ has cardinality less than $\kappa$. This means that there has to be $i<\kappa$ such that $\lnot f(y,\alpha_i)$ whenever $M\models y\in x$. But whenever $M\models x\in y$, by Item \ref{fpreserve}, it follows that $\alpha_\kappa \in f(y)$ or $f(y)=\alpha_\kappa$. Summing up, $\alpha_i$ cannot be in $f(\alpha_\kappa)$, contradicting Item \ref{fbasic}.

%   \begin{claim}
%     If $0\leq j<n$ and $\langle i_k\mid k<\omega_j\rangle$ is a sequence of ordinals below $\omega_n$ such that $M\models\alpha_{i_k}\in\alpha_{i_l}$ whenever $k<l<\omega_j$, then there is $i<\omega_n$ such that $M\models\alpha_{i_k}\in\alpha_i$ whenever $k<\omega_j$.
%   \end{claim}
%   \begin{proof}
%     Assume for a contradiction that this is not the case, as witnessed by a sequence $\langle i_k\mid k<\omega_j\rangle$. This means that whenever $i<\omega_n$, there is $k<\omega_j$ such that $M\models\alpha_i\in\alpha_{i_k}$. By the regularity of $\omega_n$, there is thus $k<\omega_j$ with $|\{m<\omega_n\mid M\models\alpha_m\in\alpha_{i_k}\}|=\omega_n$, contradicting Item \ref{small} by our first claim above.
%   \end{proof}
%   By the second claim above, we can construct a sequence $\langle i_k\mid k\le\omega_{n-1}\rangle$ of ordinals below $\omega_n$ such that $M\models\alpha_{i_k}\in\alpha_{i_l}$ whenever $k<l\le\omega_{n-1}$. But then, $|\{j<\omega_n\mid M\models \alpha_j\in\alpha_{i_{\omega_{n-1}}}\}|\ge\omega_{n-1}$, contradicting Item \ref{small} by our first claim above.
 \end{proof}

  The following small observation, which we thought is nice enough to include, is an adaptation of one part of \cite[Proposition 4.4]{kanamori}:

\begin{observation}\label{obs:fromkanamori}
  If $\omega\le\bar\delta<\delta\le\kappa$, with $\delta$ regular, and $\kappa$ is $(\delta,\kappa)$-compact, then $2^{\bar\delta}<\kappa$. % has cofinality $\ge\delta$ and satisfies $2^{<\delta}<\kappa$.
\end{observation}
\begin{proof}
  Assume that $\kappa$ is $(\delta,\kappa)$-compact.
  Assume for a contradiction that $2^{\bar\delta}\ge\kappa$ for some $\bar\delta<\delta$.
  Let $c_\alpha$ and $d^i$ be constants for $\alpha<\bar\delta$ and $i<2$.
  Let $T$ be the theory that consists of the following:
  \begin{itemize}
    \item $\{c_\alpha=d^0\lor c_\alpha=d^1\mid\alpha<\bar\delta\}$.
    \item $\{\bigvee_{\alpha<\bar\delta}(c_\alpha\ne d^{f(\alpha)})\mid f\in{}^{\bar\delta} 2\}$.
  \end{itemize}
  Then, $T$ is an $\mathcal L_{\delta,\omega}$-theory in a language of size $\bar\delta<\kappa$.
  Since $2^{\bar\delta}\ge\kappa$, $T$ is ${<}\kappa$-satisfiable. However, $T$ is not satisfiable, which is a contradiction.
  %First, assume for a contradiction that $\kappa$ has cofinality $\bar\kappa<\delta$. For every $\alpha<\bar\kappa$, let $\beta_\alpha$ be a constant symbol, let $P_\alpha$ be a predicate symbol, and let $c$ be a constant symbol. Let $<$ be a binary relation. Consider the theory $T$ which consists of the following:
  % \begin{itemize}
  %   \item $<$ is a total ordering of the domain of our structure.
  %   \item $\beta_{\alpha_0}<\beta_{\alpha_1}$ whenever $\alpha_0<\alpha_1<\bar\kappa$.
  %   \item $\forall x\ (P_{\alpha}(x) \iff \beta_\alpha\le x<\beta_{\alpha+1}$).
  %   \item $\forall x\bigvee_{\alpha<\bar\kappa}P_\alpha(x)$.
  %   \item $\{c\ne\alpha\mid\alpha<\kappa\}$.
  % \end{itemize}
\end{proof}

\section{More on compact cardinals}\label{section:more}
  
  Our next goal is to investigate the relationship between $(\delta,\lambda)$-compact cardinals and $(\delta,\lambda)$-strongly compact cardinals, as introduced by Boney and Unger \cite{boneyunger}.  
  Let~$\kappa$ be a regular cardinal, and let $A$ be any set. We let $\cP_\kappa(A)$ denote the collection of subsets of $A$ of cardinality less than $\kappa$. If $F\subseteq\cP(\cP_\kappa(A))$ is a filter, we say that $F$ is \emph{fine} in case whenever $a\in A$, $\{x\in\cP_\kappa(A)\mid a\in x\}\in F$. We say that $F$ is an \emph{ultrafilter} on a set $X$ if it measures $X$. In the terminology of \cite{boneyunger}, for uncountable cardinals $\delta\le\kappa\le\lambda$, $\kappa$ is \emph{$(\delta,\lambda)$-strongly compact} if there is a fine, ${<}\delta$-complete ultrafilter on $\cP(\cP_\kappa(\lambda))$.

\begin{proposition}\label{prop:filterextensioneasy}
  If $\delta\le\kappa\le\lambda$ are uncountable cardinals with $\delta$ regular, and $\kappa$ has the $(\delta,2^{(\lambda^{<\kappa})})$-filter extension property (in particular, this is the case if $\kappa$ is $(\delta,2^{(\lambda^{<\kappa})})$-compact by Proposition \ref{proposition:filterextensionchar}), then there is a fine, ${<}\delta$-complete ultrafilter on $\cP(\cP_\kappa(\lambda))$ -- that is, $\kappa$ is $(\delta,\lambda)$-strongly compact, or equivalently, there is an elementary embedding $j\colon V\to W$ for some transitive class $W\subseteq V$, with critical point $\crit\,j\ge\delta$\footnote{The \emph{critical point} of an elementary embedding between transitive sets or classes is the least ordinal $\alpha$ that is moved by $j$ nontrivially, that is $j(\alpha)>\alpha$. Statements like $\crit\,j\ge\delta$ are meant to include the statement that $\crit\,j$ exists.} and $D\in W$ with $j[\lambda]\subseteq D$ and $|D|<j(\kappa)$ in~$W$.
\end{proposition}
\begin{proof}
  Let $F=\{\{x\in\cP_\kappa(\lambda)\mid\gamma\in x\}\mid\gamma<\lambda\}$. Since $F$ is ${<}\kappa$-complete, we may use the $(\delta,2^{(\lambda^{<\kappa})})$-filter extension property to find $U\supseteq F$ which is ${<}\delta$-complete and measures $\cP(\cP_\kappa(\lambda))$. Note that for $U$ to extend $F$ just means that $U$ is fine. The rest of the argument is very much standard, but will be provided for the benefit of our readers, and also since we will need to use a variation of this argument later on: We use $U$ to define an equivalence relation \[f\sim g\ \iff\ \{x\in\cP_\kappa(\lambda)\mid f(x)=g(x)\}\in U\] for $f,g\colon\cP_\kappa(\lambda)\to V$, denote the equivalence class of $f\colon\cP_\kappa(\lambda)\to V$ by $[f]_U$, and define the ultrapower \[W:=\Ult(V,U)=\{[f]_U\mid f\colon\cP_\kappa(\lambda)\to V\}.\]
  Repeated application of \L o\'s's theorem yields the following: By the ${<}\omega_1$-complete\-ness of $U$, $\Ult(V,U)$ is well-founded, and we may identify it with its transitive collapse~$W$. Let $j\colon V\to W$ be defined via $j(x)=[c_x]_U$, for $x\in V$, where $c_x$ denotes the constant function with value $x$ and domain $\cP_\kappa(\lambda)$. Let $D:=[\id]_U\in W$. The embedding $j$ is nontrivial, since $\forall\alpha<\kappa\ \alpha<\ot\,D<j(\kappa)$. By the ${<}\delta$-completeness of $U$, $\crit\,j\,\ge\delta$. Moreover, it easily follows that $j[\lambda]\subseteq D$ and $|D|<j(\kappa)$ in~$W$, as desired.

  Given $j\colon V\to W$ with the above properties, let \[U=\{A\in\cP(\cP_\kappa(\lambda))\mid D\in j(A)\}.\] Then, $U$ is easily seen to be a ${<}\delta$-complete, fine ultrafilter on $\cP(\cP_\kappa(\lambda))$, using that $\crit\,j\,\ge\delta$, $j[\lambda]\subseteq D$, and that $D\in(\cP_{j(\kappa)}(j(\lambda)))^W$.
\end{proof}

\begin{proposition}\label{prop:pklequivalence}
  If $\delta\le\kappa\le\lambda$ are uncountable cardinals with $\delta$ regular and with $\lambda^{<\delta}=\lambda$, and $\kappa$ is $(\delta,\lambda)$-strongly compact then $\kappa$ is $(\delta,\lambda)$-compact.
\end{proposition}
\begin{proof}
  This follows from the second part of the proof of Proposition \ref{proposition:filterextensionchar} -- note that, using the notation of that proof, the set $\Sigma$ has size at most $\lambda^{<\delta}=\lambda$ and hence we can identify $[\Sigma]^{<\kappa}$ with $\cP_\kappa(\lambda)$. Then clearly, a fine, ${<}\delta$-complete ultrafilter on $\cP_\kappa(\lambda)$ is sufficient to proceed with the argument from that proof.
\end{proof}

We now modify the above so that we obtain an actual characterization of $(\delta,\lambda)$-compactness. Given a set $M$, we say that a filter $U\subseteq\cP(\cP_\kappa(\alpha))\cap M$ is an \emph{$M$-ultrafilter} on $\cP(\cP_\kappa(\alpha))$ if $U$ measures $\cP(\cP_\kappa(\alpha))\cap M$. We say that such $U$ is \emph{fine} if it contains the set $\{x\in\cP_\kappa(\alpha)\mid\gamma\in x\}$ as an element whenever $\gamma\in\alpha\cap M$. Furthermore, if $M\prec H(\theta)$ for some regular cardinal $\theta$, and $j\colon(M,\in)\to(N,\in_N)$ is an elementary embedding with $N$ not necessarily transitive, the \emph{critical point} $\crit\,j$ of $j$ (which may or may not exist) is the least ordinal $\gamma\in M$ with the property that $\{\bar\gamma\in N\mid\bar\gamma\in_N j(\gamma)\}\supsetneq j[\gamma\cap M]$.

\begin{proposition}\label{prop:filterextension}
  Let $\delta\le\kappa\le\lambda\le\alpha<\theta$ be uncountable cardinals with $\delta$ and~$\theta$ regular. %, and with $\lambda^{<\delta}=\lambda$.
  Then, if~$\kappa$ has the $(\delta,\lambda)$-filter extension property and $M\prec H(\theta)$ is of size~$\lambda$ with $\lambda+1\subseteq M$ and $\cP_\kappa(\alpha)\in M$, then there is a ${<}\delta$-complete, fine $M$-ultrafilter~$U$ on $\cP(\cP_\kappa(\alpha))$.

  \medskip
  
  The existence of such $U$ is equivalent to the existence of an elementary embedding $j\colon M\to N$ with $\crit\,j\ge\delta$ (but $N$ not necessarily transitive), and a set $D\in N$ with $j[\alpha\cap M]\subseteq D$ and $|D|<j(\kappa)$ in $N$.

  %In particular, if $\lambda^{<\kappa}=\lambda$, we may pick $M\supseteq[\lambda]^{<\kappa}$ and hence obtain $U$ which is a ${<}\delta$-complete, fine ultrafilter on $\cP_\kappa(\lambda)$
\end{proposition}
\begin{proof}
  We proceed similarly to the argument for Proposition \ref{prop:filterextensioneasy}. First, assume that~$\kappa$ has the $(\delta,\lambda)$-filter extension property and that $M\prec H(\theta)$ is of size~$\lambda$ with \hbox{$\lambda+1\subseteq M$} and $\cP_\kappa(\alpha)\in M$.
  We let \[F=\{\{x\in\cP_\kappa(\alpha)\mid\gamma\in x\}\mid\gamma\in\alpha\cap M\}.\] Since $F$ is ${<}\kappa$-complete and $F\subseteq M$, we may use the $(\delta,\lambda)$-filter extension property to find $U\supseteq F$ which is ${<}\delta$-complete and measures $\cP(\cP_\kappa(\alpha))\cap M$, since $M$ has size~$\lambda$.
  
  Given such $U$, we define the ultrapower \[N:=\Ult(M,U)=\{[f]_U\mid f\colon\cP_\kappa(\alpha)\to M,f\in M\}.\]
  Let $j\colon M\to N$ be defined via $j(x)=[c_x]_U$, for $x\in M$. Let $D:=[\id]_U\in N$. Repeated application of \L o\'s's theorem yields the following: The critical point $\crit \,j$ of $j$ exists, since $\forall\alpha<\kappa\ [c_\alpha]<\ot\,D<j(\kappa)$. By the ${<}\delta$-completeness of~$U$, $\crit\,j\ge\delta$. Moreover, it easily follows that $j[\alpha\cap M]\subseteq D$ and $|D|<j(\kappa)$ in $N$, as desired.

  Given $j\colon M\to N$ and $D\in N$ as in the final statement of our proposition, let 
  \[U=\{A\in\cP(\cP_\kappa(\alpha))\mid A\in M\,\land\,D\in j(A)\}.\] 
  Then, $U$ is easily seen to be a ${<}\delta$-complete, fine $M$-ultrafilter on $\cP(\cP_\kappa(\alpha))$, using that $\crit\,j\ge\delta$, $j[\alpha\cap M]\subseteq D$, and that $D\in(\cP_{j(\kappa)}(j(\alpha)))^N$.
\end{proof}

\begin{corollary}\label{cor:pklsubtleequivalence}
  Let $\delta\le\kappa\le\lambda$ be uncountable cardinals with $\delta$ regular and $\lambda^{<\delta}=\lambda$. Then, $\kappa$ is $(\delta,\lambda)$-compact if and only if for all (equivalently, for some) regular $\theta>\lambda^{<\kappa}$, whenever $M\prec H(\theta)$ is of size $\lambda$ with $\lambda+1\subseteq M$, there is~$U$ which is a ${<}\delta$-complete, fine $M$-ultrafilter on $\cP_\kappa(\lambda)$.

  The existence of such $U$ is equivalent to the existence of an elementary embedding $j\colon M\to N$ with $\crit\,j\ge\delta$ (but $N$ not necessarily transitive), and a set $D\in N$ with $j[\lambda\cap M]\subseteq D$ and $|D|<j(\kappa)$ in $N$.
\end{corollary}
\begin{proof}
  The forward direction and the equivalence with respect to elementary embeddings are immediate from Propositions \ref{prop:filterextension} and \ref{proposition:filterextensionchar} above. The reverse direction follows from the second part of the proof of Proposition \ref{proposition:filterextensionchar} -- note that, using the notation of that proof, the set $\Sigma$ has size at most $\lambda^{<\delta}=\lambda$ and hence we can identify $[\Sigma]^{<\kappa}$ with $\cP_\kappa(\lambda)$. Then, again observing that proof, a fine, ${<}\delta$-complete filter that measures a particular $\lambda$-size subset of $\cP(\cP_\kappa(\lambda))$ -- namely all the relevant sets $J_{\varphi,a}$, of which there are $\lambda^{<\delta}=\lambda$ many, is sufficient to proceed with the argument from that proof, and whenever $\theta>\lambda^{<\kappa}$ is regular, $H(\theta)$ contains all subsets of $\cP_\kappa(\lambda)$ as elements, so we may pick $M\prec H(\theta)$ to contain all the relevant sets~$J_{\varphi,a}$.
\end{proof}

\begin{appendix}
    \section[]{On the naturalness of list colouring}\label{natural}

The point of this section is to remark that for finitely many colours, colouring and list colouring are essentially the same, and that therefore, list colouring is just a way of generalizing graph colouring with finitely many colours to infinitely many colours.

\begin{observation}\label{observation:GandH}
  If $\G=\langle G,E\rangle$ is a graph and $L\colon G\to\mathcal P(\delta)$ with $\delta$ finite, then there is a graph~$\cH(\G)=\cH=\langle H,F\rangle$ with $H=G\mathrel{\dot\cup} K_\delta$ and $F\cap G^2=E$, with the following properties:
  \begin{itemize}
    \item Every $L$-colouring of $\G$ can be extended to a $\delta$-colouring of $\cH$.
    \item For any $\delta$-colouring $d$ of $\cH$, there is a permutation $\pi$ of $\delta$ such that $(\pi\circ d){\upharpoonright}G$ is an $L$-colouring of $\G$.
    \item If $\bar{\G}\le\G$, then $\cH(\bar{\G})$ is the restriction of $\cH=\cH(\G)$ to $\bar G\mathrel{\dot\cup} K_\delta$.
  \end{itemize}
\end{observation}
\begin{proof}
  Given $\G$, extend it by adding a disjoint copy of $K_\delta$, the complete graph with $\delta$-many nodes, in order to obtain $\cH$. Whenever $g\in G$ and $n\not\in L(g)$, we also add an edge between $g$ and $n$ in $\cH$.
  Now, given an $L$-colouring $c$ of $\G$, we may extend it to a $\delta$-colouring $d$ of $\cH$ by letting $d(n)=n$ for every $n\in K_\delta$. Since $n$ is connected to $g\in G$ only if $n\not\in L(g)$ it follows that $d(g)=c(g)\ne n=d(n)$, which implies that $d$ is a $\delta$-colouring of $\cH$, as desired.
  
  On the other hand, given a $\delta$-colouring $d$ of $\cH$, its restriction to $K_\delta$ is injective, hence we may find a permutation $\pi$ of $\delta$ such that $\pi\circ d{\upharpoonright}K_\delta=\id_\delta$. Then obviously, also $\pi\circ d$ is a $\delta$-colouring of $\cH$, and its restriction to $G$ is an $L$-colouring of $\G$, since whenever $g\in G$ and $n\not\in L(g)$, then there is an edge between $g$ and $n$ in $\cH$, and therefore $\pi\circ d(g)\ne n$.

  The final item is immediate from our construction of $\cH$.
\end{proof}

\begin{corollary}
  If $\delta$ is finite, then $\kappa$-compactness for $\delta$-colouring and $\kappa$-compact\-ness for $\delta$-list-colouring are equivalent.
\end{corollary}
\begin{proof}
  Since every $\delta$-colouring is an $L$-colouring for $L$ the constant function with value $\delta$, $\kappa$-compactness for $\delta$-list-colouring implies $\kappa$-compactness for $\delta$-colouring. Now assume that $\kappa$-compactness for $\delta$-colouring holds. Let $\G$ be a graph and let $L\colon G\to\cP(\delta)$ be a map such that every subgraph of $\G$ of size less than $\kappa$ has an $L$-colouring. Let $\cH$ be the graph witnessing Observation \ref{observation:GandH} with respect to~$\G$. By Observation \ref{observation:GandH} applied to subgraphs $\bar\G$ of $\G$ of size less than $\kappa$, this  means that every subgraph of $\cH$ of size less than $\kappa$ has a $\delta$-colouring, and therefore by our assumption, $\cH$ has a $\delta$-colouring. Now, again by Observation \ref{observation:GandH}, this implies that~$\G$ has an $L$-colouring, as desired.
\end{proof}

\end{appendix}

\section*{Open Questions}

The most obvious questions that naturally arise from the results of our paper seem to be those concerning the properties of our hierarchy of compact cardinals, and we ask what we think might be very interesting sample questions regarding $(\omega_1,\kappa)$-compact cardinals $\kappa$ below:

\begin{question}
    \begin{itemize}
      \item Can $\kappa=\aleph_{\omega_1}$ be $(\omega_1,\kappa)$-compact?
      \item Is an $(\omega_1,\kappa)$-compact cardinal $\kappa$ always a limit cardinal?
      \item Does an $(\omega_1,\kappa)$-compact cardinal $\kappa$ always have uncountable cofinality?
      \item Does the assumption of the existence of an $(\omega_1,\kappa)$-compact cardinal $\kappa$ have consistency strength beyond $\ZFC$?
    \end{itemize}
\end{question}

\section*{Acknowledgements}

The first author was supported by the European Research Council through the ERC Synergy Grant POCOCOP (grant agreement No.~101071674). Funded by the European Union. Views and opinions expressed are, however, those of the authors only and do not necessarily reflect those of the European Union or the European Research Council Executive Agency. Neither the European Union nor the granting authority can be held responsible for them.

The authors used AI tools by Anthropic (Claude Opus 5) and by OpenAI (GPT-5.6 Sol), both accessed in August 2026, to proofread the paper. All outputs and suggested corrections were independently verified by the authors, who take full responsibility for the content of the paper.

\bibliographystyle{amsplain}
\bibliography{references}

\end{document}